\documentclass{amsart}

\usepackage{pb-diagram}
\usepackage{amssymb}

\def\A{\mathcal {A}}

\def\S{\mathcal {S}}

\def\Fe{\mathcal {F}}
\def\R{\mathbf{R}}
\def\C{\mathbf{C}}
\def\F{\mathbf{F}}

\def\N{\mathbf{N}}

\def\K{\mathcal{K}}
\def\E{\mathcal{E}}
\def\X{\mathcal{X}}
\def\a{\alpha}

\def\V{\vartheta}
\def\w{\varpi}
\def\p{\partial}
\def\pp#1#2{\frac{\p #1}{\p #2}}
\def\PB{\left\{\cdot\,,\cdot\right\}}
\def\pb#1{\left\{#1\right\}}
\def\LB{\left[\cdot\,,\cdot\right]}
\def\Lb#1{\left[\cdot,#1\right]}
\def\lB#1{\left[#1,\cdot\right]}
\def\lb#1{\left[#1\right]}
\def\ideal#1{\langle #1\rangle}

\def\Hom{\mathop{\rm Hom}\nolimits}

\def\div{\mathop{\rm Div}\nolimits}
\def\Vect{{\mathfrak{X}}}

\def\de{\delta}

\def\v{\varphi}
\def\vn{\vec{\nabla}}
\def\we{\wedge}
\def\s{\sigma}
\def\ds{\displaystyle}
\def\Im{\mathop{\rm Im}\nolimits}
\def\Ker{\mathop{\rm Ker}\nolimits}

\def\mod{\mathop{\rm mod}\nolimits}
\def\half {{1\over 2}}  
\def\aa{{\bf a}}
\def\bb{{\bf b}}

\def\bc{\bar c}
\def\ad{\mathop{\rm ad}\nolimits}
\def\Ad{\mathop{\rm Ad}\nolimits}

\newtheorem{thm}{Theorem}[section]
\newtheorem{prp}[thm]{Proposition}
\newtheorem{cor}[thm]{Corollary}
\newtheorem{lma}[thm]{Lemma}
\theoremstyle{definition}

\newtheorem{rem}[thm]{Remark}

\newenvironment{eqn*}[1][1.5]
  {$$\renewcommand{\arraystretch}{#1}
      \begin{array}{rcl}}
      {\end{array}$$}

\newenvironment{eqn}[2][1.5]
  {\begin{equation}\label{#2}
   \renewcommand{\arraystretch}{#1}
   \begin{array}{rcl}}
  {\end{array}\end{equation}}

\begin{document}
\nocite{*}

\title{Formal deformations of Poisson structures in low dimensions}
\author{Anne Pichereau}
  \address{Centre de Recerca Matem\`atica,
           Apartat 50,
           E-08193 Bellaterra,
           Spain}
  \email{APichereau@crm.cat}
\thanks{The author was supported by
 a grant of the research network LIEGRITS, in the Marie Curie Research Training Network MRTN-CT
 2003-505078, and a grant of
 the EPDI (European Post-Doctoral Institut).}
\begin{abstract}
In this paper, we study formal deformations of Poisson
structures, especially for three families of Poisson varieties in
dimensions two and three. For these families of Poisson structures,
using an explicit basis of the second Poisson cohomology space, we
solve the deformation equations at each step and obtain a
large family of formal deformations for each Poisson structure which
we consider. With the help of an explicit formula, we show that this
family contains, modulo equivalence, all possible formal
deformations. We show moreover that, when the Poisson structure is
generic, all members of the family are non-equivalent.
\end{abstract}
\maketitle
\tableofcontents

\section{Introduction}
Poisson structures first have been introduced 
in the realm of classical mechanics, by D. Poisson. Indeed, he
discovered in $1809$
the natural symplectic structure on~$\R^{2r}$. This
structure permits one to write Hamilton's equations in a more
natural way, with positions and momenta playing symmetric roles. This
symplectic structure is, in a sense, the most simple example of a
Poisson structure and it takes the following form:
\begin{eqnarray}\label{first_poisson}
\pb{F,G}=\sum_{i=1}^r \left(\pp{F}{q_i}\pp{G}{p_i}-\pp{F}{p_i}\pp{G}{q_i}\right),
\end{eqnarray}
for smooth functions $F,G$ on $\R^{2r}$. 
In $1839$, C. Jacobi showed that this bracket satisfies the now called
Jacobi identity:
\begin{equation}\label{jacobi}
\pb{\pb{F,G},H}+\pb{\pb{G,H},F}+\pb{\pb{H,F},G}=0,
\end{equation}
thereby explaining Poisson's theorem: the bracket of two constants
of motion is a constant of motion.
In general, one defines a \emph{Poisson structure} on an associative commutative
algebra $(\A,\cdot)$, over a field $\F$, as being a
Lie algebra structure on~$\A$, $\PB: \A\times\A\to\A$, which is a
\emph{biderivation} of $\A$, i.e., satisfies the derivation
property in each of its arguments:
\begin{equation}\label{leibniz_intro}
\pb{F\cdot G,H}=F\cdot \pb{G,H} + G\cdot\pb{F,H}, 
  \quad\hbox{for all } F,G,H\in\A.
\end{equation}
A smooth manifold $M$ is said to be a \emph{Poisson manifold} if its
algebra of smooth functions $C^\infty(M)$ is equipped with a
Poisson structure.

\smallskip

Poisson structures are also inherent in quantum mechanics,
since it was observed by P. Dirac that, up to a factor $2 i \pi/h$, 
the commutator of observables, appearing in the work of
W. Heisenberg,
is the analogue of the Poisson bracket~(\ref{first_poisson}) 
of classical mechanics.
They also play an
important role in the theory of deformation quantization, which is
linked to quantum mechanics, as shown in \cite{BFFLS}. Translated in
a mathematical language, this theory is the study of 
deformations of associative, commutative algebras. 
In $1997$, 
M. Kontsevich proved that, given a Poisson manifold $(M,\PB)$, 
the equivalence classes of formal
deformations of the Poisson structure $\PB$ correspond to 
the equivalence classes of formal deformations of
the associative product of $C^\infty(M)$.
This result underlies the importance of formal deformations of Poisson
structures, which is the subject of the present paper.

\smallskip

Let $(\A,\PB)$ be a Poisson algebra over $\F$. A \emph{formal deformation} of
$\PB$ (see \cite{stern} and \cite{gutt}) 
is a map $\pi_* :\A[[\nu]]\times\A[[\nu]]\to\A[[\nu]]$ which extends $\PB$:
$$
\pi_*=\PB +\pi_1\nu +\pi_2\nu^2 + \cdots +\pi_n\nu^n+\cdots,
$$
where each map $\pi_i:\A\times\A\to\A$ is a skew-symmetric
biderivation of $\A$, and which makes $(\A[[\nu]],\pi_*)$ into a Poisson
algebra over the ring $\F[[\nu]]$, where the associative product on
$\A[[\nu]]$ is the one inherited from the initial one on $\A$. 
To simplify the notation and to emphasize the fact that the Poisson
structure $\PB$ is the first term of $\pi_*$, we also denote it by $\pi_0$. 
Notice that, similarly to the associative product, each skew-symmetric
biderivation of $\A$ (like the $\pi_i$) can be seen as a map
$\A[[\nu]]\times \A[[\nu]]\to\A[[\nu]]$, by considering its extension by
$\F[[\nu]]$-bilinearity.
In particular, such an extension of $\pi_0$ is a formal deformation of $\pi_0=\PB$,
but we stress that our goal is to consider \emph{all} formal
deformations of $\pi_0$ and not only the one obtained in this way.
If one works modulo $\nu^{n+1}$, then one speaks of an \emph{$n$-th order
  deformation}.
Deformations are always studied up to equivalence, two
formal deformations $\pi_*$ and $\pi'_*$ being \emph{equivalent} if
there exists a morphism $\Phi:(\A[[\nu]],\pi_*)\to(\A[[\nu]],\pi'_*)$
of Poisson algebras over $\F[[\nu]]$ which is the identity
modulo $\nu$.  

\smallskip

Studying deformations of a Poisson structure $\PB$ means studying
the following questions:
\begin{enumerate}
\item[$(Q_1)$] Rigidity: Do there exist non-trivial formal deformations of $\PB$?
\item[$(Q_2)$] Extendibility: Given an $n$-th order deformation of $\PB$, does it extend to an
  $(n+1)$-th order deformation? 
\item[$(Q_3)$] Formula: Is it possible to obtain an explicit formula
  for all formal / $n$-th order 
deformations of $\PB$ (up to equivalence)?
\item[$(Q_4)$] Properties: Which properties of the Poisson bracket $\PB$
  are stable under deformation? 
\end{enumerate}
In general, the deformation theory of a structure (an associative or a
Lie product, for example)
is governed by an associated cohomology, which provides some tools to
give an answer to the questions $(Q_1)$ --- $(Q_4)$.
In the particular Poisson case, the cohomology
which plays this role is \emph{Poisson cohomology} (introduced in
\cite{lichne}, see also \cite{huebs} for an algebraic approach). 
For a Poisson algebra $(\A,\pi_0=\PB)$, 
the Poisson complex (which will be explained in Paragraph \ref{Prelim}) 
is defined on the
space $\Vect^\bullet(\A)$ of all skew-symmetric multiderivations of
$\A$ (in particular, $\pi_0\in\Vect^2(\A)$). For $k\in\N$, the $k$-th Poisson
cohomology space is then denoted by $H^k(\A,\pi_0)$.

\smallskip

As we will see in Paragraph \ref{Prelim}, the third Poisson
cohomology space $H^3(\A,\pi_0)$ appears naturally in the construction
of formal deformations of $\pi_0$: a 
map of the form $\pi_*=\sum_{n\in\N}\pi_n\nu^n
:\A[[\nu]]\times\A[[\nu]]\to\A[[\nu]]$ is a formal deformation of
$\pi_0$ if and only if each $\pi_n$ ($n\in\N^*$) is a skew-symmetric
biderivation of $\A$ which satisfies a certain cohomological
equation in $H^3(\A,\pi_0)$.
That is why one refers to
$H^3(\A,\pi_0)$ as being the set of obstructions to deformations of $\pi_0$.
The second Poisson cohomology space
$H^2(\A,\pi_0)$ plays also a fundamental role in this study. Indeed, if 
$\pi_n\in\Vect^2(\A)$ is a solution of the equation mentionned above, 
then $\pi'_n = \pi_n+P$, where
$P$ is any $2$-cocycle, is also a solution, but if in particular $P$ is
a $2$-coboundary, then the corresponding $\pi'_n$ gives rise to a
($n$-th order) deformation, equivalent to the one obtained with $\pi_n$. Hence, the
choice at each step of the construction of $\pi_*$ is a choice in $H^2(\A,\pi_0)$.
The difficulty for constructing a formal deformation of $\pi_0$ can
now be explained as follows: even if, at one step, one finds a
solution for the cohomological equation mentionned above, the choice
(in $H^2(\A,\pi_0)$), which one has to 
make at this step, changes the cohomological equations (in
$H^3(\A,\pi_0)$) which one will have to solve at each of the following
steps. Now, depending on the choices that have been done previously,
the cohomological equation at one step can even be solvable or not! This
explains why, in general, it is difficult, even with a precise knowledge of the
corresponding cohomology, to answer the above questions $(Q_1)$ --- $(Q_4)$.

\smallskip

In the first part of this paper (Section \ref{sec:part_case}), we prove
a proposition which gives, for a class of Poisson structures, a system of
representatives for all formal deformations, modulo equivalence. We
formulate it here for the case of formal deformations, even if it is equally
valid for the case of $n$-th order deformations.
\begin{prp}\label{thmintro}
Let $(\A,\pi_0)$ be a Poisson algebra. 
Denote by $(\vartheta_k\in\Vect^2(\A))_{k\in \K}$, a set of
$2$-cocycles, whose cohomology classes form a basis of $H^2(\A,\pi_0)$.
Define~$\S$, the set of all
  $\aa=(a_n^k\in\F)_{{k\in\K}\atop{n\in\N^*}}$, such that, for every
  $n_0\in\N^*$, the sequence $(a_{n_0}^k)_{k\in\K}$ has a
  finite support.

\vspace*{0.5cm}
Suppose that, to each
  $\aa=(a_n^k)_{{k\in\K}\atop{n\in\N^*}}$, element of $\S$, is associated
  a sequence $\left(\Psi_n^\aa\in\Vect^2(\A)\right)_{n\in\N^*}$ of skew-symmetric
  biderivations of $\A$, satisfying:
\begin{enumerate}
\item[$\bullet$] The skew-symmetric biderivation
  $\Psi_1^\aa$ of $\A$ is zero: $\Psi_1^\aa=0$;
\item[$\bullet$] For all $n\in\N^*$, $\Psi_n^\aa$ only depends on the
$a^k_m$, with $k\in\K$ and $1\leq m \leq n-1$;
\item[$\bullet$] The skew-symmetric biderivation of $\A[[\nu]]$,
  defined by
\begin{equation*}
 \pi_*^\aa := \pi_0 + \sum_{n\in\N^*} \left(\Psi_n^\aa 
+ \sum_{k\in\K} a^k_n\;\vartheta_k \right)\nu^n,
\end{equation*}
is a formal deformation of $\pi_0$.
\end{enumerate}
Then,
\begin{enumerate}
\item[(a)] For every formal deformation $\pi_*$ of $\pi_0$, there exists
an element $\aa=(a_n^k)_{{k\in\K}\atop{n\in\N^*}}$ of $\S$, such that $\pi_*$ is
equivalent to $\pi_*^\aa$;
\item[(b)] If, in addition, the first Poisson cohomology space $H^1(\A,\pi_0)$ is
\emph{zero}, then the element $\aa\in\S$, whose existence is
mentionned in \emph{(a)}, is unique. 
\end{enumerate}
\end{prp}
We stress that not only the space $H^3(\A,\pi_0)$ 
(implicitly in the existence of the family  $(\pi_*^\aa)_{\aa\in\S}$) and
 the space $H^2(\A,\pi_0)$ (explicitly in the writing of the family
 $(\pi_*^\aa)_{\aa\in\S}$) are involved in this proposition, but also
 $H^1(\A,\pi_0)$.

\smallskip

The hypotheses in the previous proposition are strong, 
but in a second part of this paper (Section \ref{sec:dim3}), 
we will show that they are satisfied for 
several large families of Poisson structures in low dimensions. We
will do that, for each family, by using an explicit basis of
$H^2(\A,\pi_0)$ and by constructing an explicit formula for suitable
$\Psi_n^\aa$, which means solving the 
cohomological equations in $H^3(\A,\pi_0)$, that govern the
extendibility of deformations.

We first consider a big family of Poisson structures that
equip $\A:=\F[x,y,z]$, the algebra of regular (polynomial) functions
on the affine space of dimension three, $\F^3$. Indeed, to each polynomial $\v\in\A$, one can
associate a Poisson structure $\PB_\v$ on $\A$, defined
by the brackets:
\begin{equation}\label{vcrochet}
\pb{x,y}_\v = \pp{\v}{z}, \quad \pb{y,z}_\v = \pp{\v}{x}, \quad \pb{z,x}_\v = \pp{\v}{y}.
\end{equation}
Notice that this Poisson structure appears for example as the transverse 
Poisson structure to a subregular nilpotent orbit of a Lie algebra
(see \cite{PHP}).

In \cite{AP}, we have
already obtained explicit bases for the spaces $H^1(\A,\PB_\v)$ 
and $H^2(\A,\PB_\v)$, in case the
polynomial $\v$ is (weight) homogeneous with an isolated singularity, i.e.,
when the surface $\Fe_\v:\lbrace \v= 0\rbrace$ (the singular locus of
$\v$) is given by a (weight) homogeneous equation and
admits an isolated singularity (at the origin). In Section \ref{sec:dim3}, we
will use these results to show that, after a change of basis of $H^2(\A,\PB_\v)$, we are
able to exhibit a family of skew-symmetric biderivations~$\Psi_n^\aa$
of $\A$ which satisfy the conditions of Proposition
\ref{thmintro}. Since we obtain in fact an explicit formula for
every $\Psi_n^\aa$, the proposition \ref{thmintro} permits us to
write an explicit formula for all formal deformations of $\PB_\v$,
up to equivalence. More precisely, we have the following proposition 
(see Proposition \ref{prp:defo_A_v}), given here in a formal context
although it is also valid for $n$-th order deformations.
\begin{prp}\label{thmintro_v}
Let $\v\in\A=\F[x,y,z]$ be a weight homogeneous polynomial with an isolated
singularity. Consider the Poisson algebra $(\A,\PB_\v)$, 
where $\PB_\v$ is the Poisson bracket given by \emph{(\ref{vcrochet})}.
\begin{enumerate}
\item[(a)] There exist skew-symmetric biderivations $\Psi_n^\aa$ of
  $\A$ (for which we have explicit formulas), 
satisfying the hypoheses of Proposition \ref{thmintro}, for $(\A,\PB_\v)$. 
\item[(b)] The Poisson algebra $(\A,\PB_\v)$ satisfies the particular conditions of
  item~\emph{(b)} of Proposition \ref{thmintro}, unless the (weighted) degree of $\v$ equals
  the sum of the weights of the variables $x,y,z$. 
\end{enumerate}
\end{prp}
At that point, we have obtained a clear answer to the question $(Q_1)$ and
$(Q_3)$ above (questions of rigidity and formula). Because Proposition \ref{thmintro} is
also true for $n$-th order deformations, we also have an answer to the
question $(Q_2)$ of extendibility, which is the following: \textit{Every $n$-th order
deformation of $\PB_\v$ extends to a $(n+1)$-th order deformation}
(Corollary \ref{extension_v}).

Finally, using the explicit formula mentionned above (for all formal deformations of the
bracket $\PB_\v$), we will also give a partial answer to the question
$(Q_4)$ of the properties stable under deformation, with the following
result: 
\textit{The formal deformations of $\PB_\v$ all admit formal Casimirs}, for which
we also have an explicit writing (Corollary \ref{casimir_v}).

\smallskip

The polynomial $\v\in\F[x,y,z]$ is a Casimir for the
Poisson structure $\PB_\v$, so that this Poisson structure restricts
to a Poisson structure $\PB_{\A_\v}$, 
on the quotient algebra $\A_\v := \frac{\F[x,y,z]}{\ideal{\v}}$, which
is the algebra of regular functions on the surface $\Fe_\v:\lbrace
\v=0\rbrace\subset \F^3$. The deformations of the Poisson structure $\PB_{\A_\v}$ 
are studied in Paragraph \ref{sbssec:A_v}. In fact, under the previous
hypotheses on $\v$, the cohomological equations mentionned above are
in this case trivial and this fact, together with an explicit basis of
the second Poisson cohomology space (obtained in \cite{AP}), permit us to give an
explicit formula for all formal deformations of $\PB_{\A_\v}$, up to
equivalence (see Proposition \ref{prp:defo_Av}).

\bigskip

\textit{Acknowledgments:} 
I wish to take this opportunity to thank P. Vanhaecke for drawing my
attention to these questions about deformations and for useful
discussions about this subject. 
I also would like to thank 
B. Fresse, C. Laurent-Gengoux, M. Penkava, R. Yu and N. T. Zung for valuable
conversations which contributed to this paper.
The hospitality of the University of Antwerp and of the 
CRM (Centre de Recerca Matematic\`a, Barcelona) 
is also greatly acknowledged.

\section{Conditions for a system of representatives for all formal
  deformations}
\label{sec:part_case}

In this part, we want to show Proposition \ref{thmintro}, anounced
in the introduction. To do that, we will need several intermediate
results, which will be proved in an elementary way, in the sense that our
proofs will only need the properties of the Schouten bracket
and the definition of the Poisson cohomology,
that are recalled in the first paragraph \ref{Prelim}.

\subsection{Preliminaries}
\label{Prelim}

In this paper, $\F$ is an arbitrary field of characteristic zero.
We recall that a \emph{Poisson
structure} $\PB$ (which is also denoted by $\pi_0$) on an associative
commutative algebra $(\A,\cdot)$ is a skew-symmetric biderivation of
$\A$, i.e., a map $\PB : \we^2\A\to\A$ satisfying the derivation
property:
\begin{equation}\label{leibniz}
\pb{FG,H} =F\pb{G,H} + G\pb{F,H}, \hbox{ for all } F,G,H\in\A,
\end{equation}
(where $FG$ stands for $F\cdot G$), 
which is also a Lie structure on $\A$, i.e., which satisfies the Jacobi identity.

We denote by $\F^\nu$ the ring of all formal power series in an
indeterminate $\nu$ and with coefficients in $\F$, i.e.,
$\F^\nu:=\F[[\nu]]$. We will also consider the $\F^\nu$-vector
space $\A^\nu:=\A[[\nu]]$ of
all formal power series in $\nu$, with coefficients in $\A$. The
associative commutative product ``$\cdot$'', defined on $\A$, is naturally extended to
an associative, commutative product on $\A^\nu$, still denoted by
``$\cdot$''. In the following, any map defined on $\A$ or on $\we^\bullet
\A$ is possibly seen as a map on $\A^\nu$ or $\we^\bullet\A^\nu$ (the
exterior algebra of the $\F^\nu$-vector
space $\A^\nu$), which means that we consider its natural extension by
$\F^\nu$-linearity.
We point out that an $\F^\nu$-$k$-linear map
$\psi:(\A^\nu)^k\to\A^\nu$ can be written as : $\psi=\psi_0 +
\psi_1\nu + \cdots + \psi_n\nu^n + \cdots $, where each $\psi_i$ is a
$k$-linear map $\A^k\to\A$. This permits us to write a natural
isomorphism of $\F^\nu$-vector spaces $\Hom((\A^\nu)^k,\A^\nu) \simeq
\Hom(\A^k,\A)[[\nu]]$.

A \emph{formal deformation} of a Poisson structure $\pi_0$ on $\A$ is a Poisson structure on the
$\F^\nu$-algebra $\A^\nu$, that extends the initial
Poisson structure. In other words, it is given by a map 
$\pi_* : \A^\nu\times\A^\nu \to \A^\nu$ satisfying the Jacobi
identity and of the form :
$$
\pi_* = \pi_0 +\pi_1\nu +\cdots +\pi_n\nu^n +\cdots,
$$
where the $\pi_i$ are skew-symmetric biderivations of $\A$.
If one works modulo $\nu^{n+1}$ (for $n\in\N$), 
i.e., if one replaces the $\F^\nu$-algebra $\A^\nu$ with the $\F^\nu /
\ideal{\nu^{n+1}}$-algebra~$\A^\nu/\ideal{\nu^{n+1}}$ in the previous definition,
one then speaks of \emph{$n$-th order deformation} of~$\pi_0$.

In order to have some tools to study formal (or $n$-th order)
deformations of Poisson structures, we
recall the notion of \emph{Poisson cohomology}. The Poisson complex is defined as
 follows: the space of all Poisson
cochains is $\Vect^\bullet(\A):=\bigoplus_{k\in\N}\Vect^k(\A)$, where
$\Vect^0(\A)$ is $\A$ and, for all $k\in\N^*$, $\Vect^k(\A)$ denotes
the space of all skew-symmetric $k$-derivations of $\A$, i.e., the skew-symmetric
$k$-linear maps $\A^k\to\A$ that satisfy the derivation property
(\ref{leibniz}) in each of their arguments.
Then, the Poisson coboundary operator
$\delta^k_{\pi_0}:\Vect^k(\A)\to\Vect^{k+1}(\A)$ is given by the formula
$$
\delta^k_{\pi_0} := -\Lb{\pi_0}_S,
$$ 
where $\LB_S:\Vect^p(\A)\times\Vect^q(\A)\to\Vect^{p+q-1}(\A)$ is the so-called
Schouten bracket (see~\cite{PLV}). The Schouten bracket is a graded
Lie bracket that generalizes the commutator of derivations and that is
a graded biderivation with respect to the wedge product of multiderivations. It is
defined, for $P\in\Vect^p(\A)$, $Q\in\Vect^q(\A)$ and
$F_1,\dots,F_{p+q-1}\in\A$, by:
\begin{eqnarray}\label{schou_exp}
\lefteqn{\lb{P,Q}_S[F_1,\dots,F_{p+q-1}]}\hspace{0cm}\nonumber\\
    \quad&=&\sum_{\s\in S_{q,p-1}}\epsilon(\s)P\left[Q[F_{\s(1)},\dots,F_{\s(q)}],
                  F_{\s(q+1)},\dots,F_{\s(q+p-1)}\right]\\
    &-&(-1)^{(p-1)(q-1)}\sum_{\s\in S_{p,q-1}}\epsilon(\s)Q\left[P[F_{\s(1)},\dots,F_{\s(p)}],
                  F_{\s(p+1)},\dots,F_{\s(p+q-1)}\right],\nonumber
\end{eqnarray}
where, for $k,\ell\in\N$, $S_{k,\ell}$ denotes the set of all
permutations $\s$ of $\lbrace 1,\dots,k+\ell\rbrace$, satisfying
$\s(1)<\cdots<\s(k)$ and $\s(k+1)<\cdots<\s(k+\ell)$, while
$\epsilon(\s)$ denotes the signature of such a permutation $\s$.
Notice that, similarly to the case of multilinear maps, it is easy to
verify that, for all $k\in\N$, the $\F^\nu$-vector space
$\Vect^k(\A^\nu)$ of all skew-symmetric $k$-derivations of the
associative algebra $(\A^\nu,\cdot)$ is isomorphic to
$\Vect^k(\A)[[\nu]]$. Indeed, every $\psi\in\Vect^k(\A^\nu)$ can be
written as $\psi=\psi_0+\psi_1\nu+\cdots +\psi_n\nu^n +\cdots$, where
each $\psi_i$ is an element of $\Vect^k(\A)$.
In the following, the Schouten bracket will often be considered as a
map, defined on $\Vect^\bullet(\A^\nu)\times\Vect^\bullet(\A^\nu)$,
with the meaning that it is simply extended by
$\F^\nu$-bilinearity and still denoted by $\LB_S$. The map $\LB_S$
then obtained is in fact a graded Lie algebra structure on
$\Vect^\bullet(\A^\nu)$, that could also be defined by a formula
analogous to (\ref{schou_exp}).

It is then easy and useful to see that, given a skew-symmetric biderivation
 $\pi\in~\Vect^2(\A)$, the Jacobi identity for $\pi$ is equivalent to the
 equation $\lb{\pi,\pi}_S=~0$.
Then, because of the graded Jacobi identity satisfied by $\LB_S$ and the
 fact that $\lb{\pi_0,\pi_0}_S=~0$, the operator $\delta_{\pi_0}$ is a coboundary
operator, leading to the Poisson cohomology spaces associated to $(\A,\pi_0)$ and defined by
$ H^k(\A,\pi_0) := \Ker \delta_{\pi_0}^k\big/
\Im\delta_{\pi_0}^{k-1}$, for $k\in\N^*$.
Elements of $Z^k(\A,\pi_0):=\Ker \delta_{\pi_0}^k\subseteq
\Vect^k(\A)$ are the (Poisson) \emph{$k$-cocycles}, while elements of
$B^k(\A,\pi_0):=\Im \delta_{\pi_0}^{k-1}\subseteq \Vect^k(\A)$ are
the (Poisson) \emph{$k$-coboundaries}. 

Moreover, given a map $\pi_* = \pi_0 +\pi_1\nu +\cdots +\pi_n\nu^n
+\cdots : \A^\nu\times\A^\nu\to\A^\nu$ where for all $i\in\N$,
$\pi_i\in\Vect^2(\A)$ is a skew-symmetric biderivation of $\A$, we
have that $\pi_*$ is a formal deformation of $\pi_0$, if and only if,
$\lb{\pi_*,\pi_*}_S=0$, i.e., if and only if, for all $n\in\N$,
\begin{eqnarray}\label{eqn:eq_cohom}
\delta_{\pi_0}^2(\pi_{n+1})  = \frac{1}{2}
  \renewcommand{\arraystretch}{0.7} \sum_{\begin{array}{c}\scriptstyle
  i+j=n+1\\\scriptstyle i,j\geq1\end{array}}\lb{\pi_i,\pi_j}_S.
\end{eqnarray}
Similarly, an $n$-th order deformation $\pi_{(n)}=\pi_0 + \pi_1\nu
+\cdots +\pi_n\nu^n$ will extend to an $(n+1)$-th order deformation
$\pi_{(n+1)}=\pi_{(n)}+\pi_{n+1}\nu^{n+1}$, if and only if,
there exists $\pi_{n+1}\in\Vect^2(\A)$, solution of the previous equation (\ref{eqn:eq_cohom}).

\subsection{Equivalent formal deformations}
\label{sbsec:equiv}

In this paragraph, for an arbitrary Poisson algebra $(\A,\pi_0)$, we
 write a formula, involving only the
Schouten bracket $\LB_S$, for the elements of the equivalence
class of a given formal deformation of $\pi_0$. 

First, we recall the notion of equivalence for deformations of $\pi_0$.
Two formal deformations $\pi_*$ and $\pi'_*$ of $\pi_0$ are said to be
\emph{equivalent} if there exists
an $\F^\nu$-linear map $\Phi:(\A^\nu,\pi_*)\to(\A^\nu,\pi_*')$ that is
a Poisson morphism
and which is such that $\Phi$ is the identity modulo $\nu$. In this case, 
we write $\pi_* \sim\pi'_*$ and we call $\Phi$ an \emph{equivalence morphism from $\pi_*$ to
$\pi_*'$}. In other words,
an $\F^\nu$-linear map $\Phi:\A^\nu\to\A^\nu$ is an equivalence
morphism from $\pi_*$ to $\pi_*'$, if and only if, it is a morphism of
associative algebras, equal to the identity modulo $\nu$ and which satisfies
$$
\Phi(\pi_*[F,G])=\pi'_*[\Phi(F),\Phi(G)],
$$ 
for all $F,G\in\A$ (and therefore, for all $F,G\in\A^\nu$). Notice
that, of course, if $\Phi$ is an equivalence morphism from $\pi_*$ to
$\pi_*'$, then $\Phi^{-1}$ is an equivalence
morphism from $\pi_*'$ to $\pi_*$. 
Similarly, one defines the notion of equivalence for
$n$-th order deformations, by replacing $\F^\nu$ with $\F^\nu/\ideal{\nu^{n+1}}$ and
$\A^\nu$ with $\A^\nu/\ideal{\nu^{n+1}}$ in the previous definition. 

Now, it is straightforward to show that the exponential map gives a bijection
between the space $\Vect^1_0(\A^\nu):=\lbrace \xi=\sum_{k\geq
  1}\xi_k\nu^k \mid \xi_k\in\Vect^1(\A), k\in\N^*\rbrace$ and the
space of all automorphisms of $\A^\nu$ that are equal to the identity
modulo $\nu$. This permits us to write an equivalence morphism $\Phi$ 
between two formal deformations of $\pi_0$ as the image of an element
of $\Vect^1_0(\A^\nu)$, by the exponential map. This implies that
the equivalence classes of formal deformations of $\pi_0$ can be
defined as the equivalence classes of the action, defined as follows, of
$\Vect^1_0(\A^\nu)$ on the formal deformations of~$\pi_0$. 
For a formal deformation $\pi_*$ of $\pi_0$ and for
$\xi\in\Vect^1_0(\A^\nu)$, we define the action $\xi\cdot\pi_*$,
mentionned above, by :
\begin{equation}\label{eqn:action}
\xi\cdot\pi_* [F,G] := e^{\xi}\left(\pi_*\left[e^{-\xi}(F),e^{-\xi}(G)\right]\right),
\end{equation}
for all $F,G\in\A$.
It is then possible to show the following equality :
$$
\xi\cdot\pi_* = e^{\ad_\xi}(\pi_*),
$$
where $\ad_\xi:=\lB{\xi}_S$. This equality involves two notions of
exponential:
\begin{enumerate}
\item[(a)] $e^\xi := \sum\limits_{k\in\N}\frac{1}{k!} \xi^k : \A^\nu \to \A^\nu$,
\item[(b)] $e^{\ad_\xi} := \sum\limits_{k\in\N}\frac{1}{k!}
  (\ad_\xi)^k : \Vect^\bullet(\A^\nu) \to \Vect^\bullet(\A^\nu)$,
\end{enumerate}
for $\xi = \xi_1\nu + \xi_2\nu^2 +\cdots +
\xi_n\nu^n+\cdots\in\Vect^1_0(\A^\nu)$, with $\xi_i\in\Vect^1(\A)$,
for all $i\in\N^*$, and where $\ad_\xi$ is the graded
derivation (of degree $0$) of the associative algebra
$(\Vect^\bullet(\A^\nu),\wedge)$, $\ad_\xi = \lB{\xi}_S$.
In fact, we have 
\begin{eqnarray*}
\lefteqn{e^\xi\left(\pi_*\left[e^{-\xi}(F),e^{-\xi}(G)\right]\right) =}\\ 
\pi_*[F,G] &+& \sum_{k\in\N^*} \renewcommand{\arraystretch}{0.7}
\sum_{\begin{array}{c}\scriptstyle r,s,t\in\N\\ \scriptstyle r+s+t=k\end{array}}
  (-1)^{s+t}\frac{1}{r!}\frac{1}{s!}\frac{1}{t!}
\,\xi^r\left(\pi_*[\xi^s(F),\xi^t(G)]\right) = \\
\pi_*[F,G] &+& \sum_{k\in\N^*} \frac{1}{k!}(\ad_\xi)^k(\pi_*)[F,G] =\\
&& e^{\ad_\xi}(\pi_*)[F,G],
\end{eqnarray*}
where the second equality is easily proved by induction on $k\in\N^*$.
Notice that this action of $\Vect^1_0(\A^\nu)$ can be extended on the
space $\Vect^\bullet(\A^\nu)$ of all skew-symmetric
multiderivations of $\A^\nu$ and then, for any
$Q\in\Vect^\bullet(\A^\nu)$, the formula $\xi\cdot Q = e^{\ad_\xi}(Q)$
still holds. 

This result can be seen as an analog of the well-known
formula that links the adjoint representation $\Ad$ of a Lie group 
${\bf G}$ on its Lie algebra $\mathfrak g$ and the
adjoint action $\ad$ of ${\bf G}$ on $\mathfrak g$: 
$\Ad_{e^\xi} = e^{\ad_\xi}$, for all $\xi\in\mathfrak g$.

Finally, we have obtained the following:
\begin{lma}\label{lma:equiv}
Let $(\A,\pi_0)$ be a Poisson algebra. 
Let $\pi_*$ be a formal deformation of~$\pi_0$. The formal
deformations of $\pi_0$ that are equivalent to $\pi_*$ are precisely the maps $\pi'_*$
of the form
\begin{eqnarray*}
\pi'_* &=& e^{\ad_\xi}(\pi_*) \\
\Bigg(&=& \pi_* + \sum_{k\in\N^*} \frac{1}{k!} 
\underbrace{\lb{\xi,\lb{\xi,\dots\,,\lb{\xi,\pi_*}_S\dots}_S}_S}
_{\hbox{$k$ brackets}}\Bigg)
\end{eqnarray*}
with $\xi\in\Vect^1_0(\A^\nu)$ (i.e., $\xi=\sum_{k\geq 1}\xi_k\nu^k$,
with $\xi_k\in\Vect^1(\A)$, for all $k\in\N^*$).
\end{lma}
Notice that there is an analogous result if one considers rather the
formal deformations of an associative commutative or a Lie product,
but then, the $\xi_k$ do not have to be derivations of $\A$ and the
Schouten bracket has to be replaced by the corresponding graded Lie
algebra structure on the cochains of the Hochschild (Gerstenhaber
bracket) or Chevalley-Eilenberg cohomology (Nijenhuis-Richardson bracket).

\subsection{Deformations of Poisson structures in a good case}
%
%
%

In this paragraph, we prove a proposition which gives, for a
certain class of Poisson structures, all formal deformations up to
equivalence. The hypotheses involved in this proposition are strong,
but we will be able to apply this result to big
families of Poisson algebras that we will consider in
Section \ref{sec:dim3} of this paper. 
\begin{prp}
 \label{magic}
Let $(\A,\pi_0)$ be a Poisson algebra. 
Suppose that $(\vartheta_k\in\Vect^2(\A))_{k\in \K}$ is a set of
$2$-cocycles, whose cohomology classes form an $\F$-basis of
$H^2(\A,\pi_0)$ and define $\S$, the set of all 
$\aa=(a_n^k\in\F)_{{k\in\K}\atop{n\in\N^*}}$, such that, for every
  $n_0\in\N^*$, the sequence $(a_{n_0}^k)_{k\in\K}$ has a
  finite support.

Suppose that we have a family $(\pi_*^\aa)_{\aa\in\S}$ of formal
deformations of the Poisson structure $\pi_0$, 
indexed by the elements $\aa=(a_n^k)_{{k\in\K}\atop{n\in\N^*}}$
of $\S$, and of the form:
\begin{equation}
 \label{defo_family}
 \pi_*^\aa = \pi_0 + \sum_{n\in\N^*} \left(\Psi_n^\aa 
+ \sum_{k\in\K} a^k_n\;\vartheta_k \right)\nu^n,
\end{equation}
where, for all $n\in\N^*$, $\Psi_n^\aa\in\Vect^2(\A)$ is a 
skew-symmetric biderivation of $\A$, depending only on the 
$a_m^k$, where $k\in\K$ and $1\leq m < n$; and
$\Psi_1^\aa=0$. Then, we have the following:
\begin{enumerate}
\item[(a)] For any formal deformation $\pi_*$ of $\pi_0$, there
  exists an element 
$\aa=(a_n^k)_{{k\in\K}\atop{n\in\N^*}}$ of $\S$, such that $\pi_*$ is
equivalent to $\pi_*^\aa$;
\item[(b)] For any $m$-th order deformation $\pi_{(m)}$ of $\pi_0$
($m\in\N^*$), there exists an element
$\aa=(a_n^k)_{{k\in\K}\atop{n\in\N^*}}$ of $\S$, such that $\pi_{(m)}$ is
equivalent to $\pi_*^\aa$ modulo $\nu^{m+1}$, i.e., in $\A^{\nu}/\ideal{\nu^{m+1}}$.
\end{enumerate}
\end{prp}
\begin{proof}
Let $\pi_* =\pi_0+ \sum\limits_{k\in\N^*} \pi_k\nu^k$ be an arbitrary
formal deformation of $\pi_0$. According to Lemma \ref{lma:equiv}, the
existence of an element $\aa=(a_n^k)_{{k\in\K}\atop{n\in\N^*}}\in\S$,
such that $\pi_* \sim \pi_*^\aa$, is equivalent to the existence of an
element $\xi=\sum_{k\in\N^*} \xi_k\nu^k\in\Vect^1_0(\A^\nu)$ such that 
$$
\pi_*=e^{\ad_\xi}(\pi_*^\aa).
$$
In order to simplify the notation, for every $\aa\in\S$ and every
$\xi\in\Vect^1_0(\A^\nu)$, we write $\pi_*^{\aa,\xi} : =
e^{\ad_\xi}(\pi_*^\aa)$ and $\pi_*^{\aa,\xi} = \pi_0 + \sum_{i\in\N^*}
\pi_i^{\aa,\xi}\,\nu^i$, $\pi^\aa_*=\pi_0 +
\sum_{i\in\N^*}\pi_i^\aa\,\nu^i$, 
with $\pi_i^{\aa,\xi},\pi_i^\aa\in\Vect^2(\A)$, for every $i\in\N^*$.

We will then show that, for every $N\in\N^*$, there exist $a^k_1,a^k_2,\dots,a^k_N\in\F$,
for $k\in\K$ (such that, for every $1\leq i\leq N$, only a finite
number of $a^k_i$ are non-zero) and $\xi_1,\dots,\xi_N\in\Vect^1(\A)$ such
that 
\begin{equation}
\pi_* = \pi_*^{\aa_{(N)},\xi_{(N)}}= e^{\ad_{\xi_{(N)}}}(\pi_*^{\aa_{(N)}})   \; \mod \nu^{N+1},
\end{equation}
where $\aa_{(N)} := (a^k_1,a^k_2,\dots,a^k_N,0,0,\dots)_{k\in\K}
= (b^k_n)_{{k\in\K}\atop{n\in\N^*}}\in\S$ with $b^k_n = a^k_n$, for $1\leq
n\leq N$ and $b^k_n =0$
as soon as $n> N$ and $\xi_{(N)}:=\xi_1\nu +\cdots +
\xi_N\nu^N\in\Vect^1_0(\A^\nu)$. We will do that by induction on
$N\in\N^*$. 
Notice that, in order to prove the second point of the proposition,
with $m$-th order deformations of $\pi_0$, we just have to use
the same proof, but only for $1\leq N\leq m$.

First of all, suppose that $N=1$. We know, according to
(\ref{eqn:eq_cohom}), that $\delta^2_{\pi_0}(\pi_1)=0$, so that, by
definition of the $\vartheta_k$, there exist
$a^k_1\in\F$, for all $k\in\K$ (with only a finite number of non-zero $a^k_1$), and $\xi_1\in\Vect^1(\A)$ such that:
$$
\pi_1 = \sum_{k\in\K} a^k_1\vartheta_k - \delta^1_{\pi_0}(\xi_1).
$$ 
Denoting by 
$\aa_{(1)}:=(a^k_1,0,0,\dots)_{k\in\K}\in\S$, and
$\xi_{(1)}:=\xi_1\nu\in\Vect^1_0(\A^\nu)$, we have:
$$
\pi_*^{\aa_{(1)}} = \pi_0 + \sum_{k\in\K}a^k_1\vartheta_k \nu \;\mod\nu^2,
$$
hence the following
equalities in $\A^\nu/\ideal{\nu^2}$:
\begin{eqnarray*}
\pi_*^{\aa_{(1)},\xi_{(1)}} 
&=& e^{\ad_{\xi_{(1)}}}(\pi_*^{\aa_{(1)}}) \;\mod\nu^2\\
&=& \pi_0 + \sum_{k\in\K}a^k_1\vartheta_k \nu
      +\lb{\xi_1,\pi_0}_S\nu \;\mod\nu^2 \\
&=& \pi_0 + \pi_1\nu\;\mod\nu^2,
\end{eqnarray*}
which achieves the case $N=1$. Suppose now $N\geq 1$ and assume the
existence of elements $a^k_n \in\F$,
for $k\in\K$ and $1\leq n\leq N$ (with, for every $1 \leq n_0\leq
N$, only a finite number of non-zero $a^k_{n_0}$) and the
existence of $\xi_1,\dots,\xi_N\in\Vect^1(\A)$, satisfying: 
\begin{equation}
\pi_* = \pi_*^{\aa_{(N)},\xi_{(N)}}\; \mod \nu^{N+1},
\end{equation}
where $\aa_{(N)} :=(a^k_1,a^k_2,\dots,a^k_N,0,0\dots)_{k\in\K}\in\S$
and $\xi_{(N)}:=\xi_1\nu +\cdots + \xi_N\nu^N\in\Vect^1_0(\A^\nu)$.
We want to show that this equality can be extended to the rank
$N+1$, with some $a^k_{N+1}\in\F$, $k\in\K$ and $\xi_{N+1}\in\Vect^1(\A)$.
As, by induction hypothesis, we have $\pi_i =
\pi_i^{\aa_{(N)},\xi_{(N)}}$, for all $1\leq i \leq N$,
Equation (\ref{eqn:eq_cohom}) implies 
$$
\delta^2_{\pi_0}\left(\pi_{N+1}\right) =
\delta^2_{\pi_0}\left(\pi_{N+1}^{\aa_{(N)},\xi_{(N)}}\right),
$$
so that there exist $a^k_{N+1}\in\F$, for $k\in\K$ (among which only a
finite number are non-zero) and
$\xi_{N+1}\in\Vect^1(\A)$, such that 
\begin{equation}\label{eqn:almost_defo}
\pi_{N+1} = \pi_{N+1}^{\aa_{(N)},\xi_{(N)}} +
\sum_{k\in\K} a^k_{N+1} \vartheta_k - \delta^1_{\pi_0}\left(\xi_{N+1}\right).
\end{equation}
Similarly to previously, let us denote by
$\aa_{(N+1)} := (a^k_1,a^k_2,\dots,a^k_{N+1},0,0\dots)\in\S$
and $\xi_{(N+1)}:=\xi_1\nu +\cdots +
\xi_{N+1}\nu^{N+1}\in\Vect^1_0(\A^\nu)$.
By definition of the $\Psi_n^\bb$, for
$\bb\in\S$ and of the elements $\aa_{(N+1)}$ and
$\aa_{(N)}$, the skew-symmetric biderivation
$\Psi_{N+1}^{\aa_{(N+1)}}$ depends only on the $a_m^k$, with $k\in\K$
and $1\leq m <N+1$, i.e., only on $\aa_{(N)}$ and
$\Psi_{N+1}^{\aa_{(N+1)}}=\Psi_{N+1}^{\aa_{(N)}}$. By definition of
the formal deformations of the form $\pi_*^\bb$, we then have:
\begin{equation*}
\pi_{N+1}^{\aa_{(N+1)}} 
= \Psi_{N+1}^{\aa_{(N)}} + \sum_{k\in\K} a^k_{N+1}\;\vartheta_k
= \pi_{N+1}^{\aa_{(N)}} + \sum_{k\in\K} a^k_{N+1}\;\vartheta_k.
\end{equation*}
Then, using 
the fact that $\pi_\ell^{\aa_{(N+1)}} = \pi_\ell^{\aa_{(N)}}$, for all
$\ell<N+1$, we also have:
\begin{eqnarray*}
\lefteqn{\pi_{N+1}^{\aa_{(N+1)},\xi_{(N+1)}} = \pi_{N+1}^{\aa_{(N+1)}}} \\
           &+& \sum_{r\in\N^*} \frac{1}{r!}
   \sum_{\ell= 0}^N\! \renewcommand{\arraystretch}{0.7}
\sum\limits_{\begin{array}{c}\scriptstyle i_1+\cdots+i_r+\ell=N+1\\
 \scriptstyle 1\leq i_1,\dots,i_r\leq N+1\end{array}}\!
   \lb{\xi_{i_1},\lb{\xi_{i_2},\dots\,,
       \lb{\xi_{i_r},\pi_\ell^{\aa_{(N+1)}}}_S\dots}_S}_S\\
 &=& \pi_{N+1}^{\aa_{(N+1)}} + \lb{\xi_{N+1},\pi_0}_S\\
       &+& \sum_{r\in\N^*} \frac{1}{r!}
   \sum_{\ell= 0}^N\! \renewcommand{\arraystretch}{0.7}
\sum\limits_{\begin{array}{c}\scriptstyle i_1+\cdots+i_r+\ell=N+1\\
 \scriptstyle 1\leq i_1,\dots,i_r\leq N\end{array}}\!
 \lb{\xi_{i_1},\lb{\xi_{i_2},\dots\,,\lb{\xi_{i_r},\pi_\ell^{\aa_{(N+1)}}}_S\dots}_S}_S\\
 &=& \pi_{N+1}^{\aa_{(N)}} + \sum_{k\in\K} a^k_{N+1}\;\vartheta_k 
   + \lb{\xi_{N+1},\pi_0}_S \\
   &+& \sum_{r\in\N^*} \frac{1}{r!}
   \sum_{\ell= 0}^N\; \renewcommand{\arraystretch}{0.7}
\sum\limits_{\begin{array}{c}\scriptstyle i_1+\cdots+i_r+\ell=N+1\\
 \scriptstyle 1\leq i_1,\dots,i_r\leq N\end{array}}
   \lb{\xi_{i_1},\lb{\xi_{i_2},\dots\,,\lb{\xi_{i_r},\pi_\ell^{\aa_{(N)}}}_S\dots}_S}_S\\
 &=& \pi_{N+1}^{\aa_{(N)},\xi_{(N)}} 
   + \sum_{k\in\K} a^k_{N+1}\;\vartheta_k + \lb{\xi_{N+1},\pi_0}_S\\
 &=& \pi_{N+1},
\end{eqnarray*}
where, in last step, we used Equation (\ref{eqn:almost_defo}). This
achieves the proof.
\end{proof}

\subsection{The case of $H^1(\A,\pi_0)=\lbrace 0\rbrace$}
\label{subsec:roleh1}

In this paragraph, we study equivalent formal deformations of a 
Poisson structure, under the assumption that the
first cohomology space $H^1(\A,\pi_0)$ is zero. We will in fact study
in Section \ref{sec:dim3} of this paper, a family of Poisson structures, for
which this space is generically zero. We use the result given in this paragraph.
Before giving this result, we need the following
\begin{lma}\label{lma:H1H1formal}
Let $(\A,\pi_0)$ be a Poisson algebra and let $\pi_*$ be a formal
deformation of $\pi_0$. Suppose that the first Poisson cohomology space, 
associated to the initial Poisson algebra, is zero: 
$$
H^1(\A,\pi_0) = \lbrace 0\rbrace.
$$
Then, we have the following:
\begin{enumerate}
\item[(a)] The first Poisson cohomology space, associated to the Poisson
algebra $(\A^\nu,\pi_*)$, is zero:
$$
H^1(\A^\nu,\pi_*)=\lbrace 0\rbrace;
$$
\item[(b)] For all $N\in\N^*$, the first Poisson cohomology space, 
associated to the Poisson algebra 
$\left(\A^\nu/\ideal{\nu^N},\pi_* \,\mod\nu^N\right)$, is zero:
$$
H^1\left(\A^\nu/\ideal{\nu^N},\pi_* \,\mod\nu^N\right) = \lbrace 0\rbrace. 
$$
\end{enumerate}

\end{lma}
\begin{proof}
Let $(\A,\pi_0)$ be a Poisson algebra such that $H^1(\A,\pi_0)=\lbrace
0\rbrace$ and let $\pi_*=\pi_0 + \sum_{i\in\N^*}\pi_i\nu^i$ be a formal deformation of $\pi_0$.
Suppose that $\psi\in\Vect^1(\A^\nu)$ 
is a $1$-cocycle for the Poisson algebra $(\A^\nu,\pi_*)$. It means that we have 
\begin{equation}\label{eq:formal1cocycle}
\lb{\psi,\pi_*}_S = 0.
\end{equation}
We write $\psi=\sum_{i\in\N}\psi_i\nu^i$, with
$\psi_i\in\Vect^1(\A)$ for all $i\in\N$.
Now, in order to prove the first part of the lemma, we will show that
for all $m\in\N^*$, there exist 
$h_0,h_1,\dots,h_{m-1}\in\A$, satisfying 
\begin{equation}\label{eq:formal1coboundaryk}
\psi + \lb{h_0+h_1\nu+\cdots+h_{m-1}\nu^{m-1},\pi_*}_S = 0 \,\mod\nu^m. 
\end{equation}
Indeed, denoting by $H\in\A^\nu$ the element
$H=\sum_{i\in\N}h_i\nu^i$, 
this shows that $\psi = -\lb{H,\pi_*}_S = \delta^1_{\pi_*}(H)$ 
is a $1$-coboundary for the Poisson algebra $(\A^\nu,\pi_*)$ and it
permits us to conclude that $H^1(\A^\nu,\pi_*)=\lbrace 0\rbrace$.
Notice that in order to prove the second part of the lemma, it
suffices to do exactly the same proof but only for $1\leq m< N$.

By induction, we will show the equality (\ref{eq:formal1coboundaryk}), for all
$m\in\N^*$. First of all, let us consider
the case $m=1$. In fact, (\ref{eq:formal1cocycle}) implies in
particular that $0 = \lb{\psi,\pi_*}_S \,\mod\nu = \lb{\psi_0,\pi_0}_S =
-\delta^1_{\pi_0}(\psi_0)$. As $H^1(\A,\pi_0)=\lbrace 0\rbrace$, we then
obtain the existence of an element $h_0\in\A$, such that $\psi_0 =
\delta^0_{\pi_0}(h_0)=-\lb{h_0,\pi_0}_S$, 
which is exactly (\ref{eq:formal1coboundaryk}), for $m=1$.

Now, suppose $m\geq 1$ and that we have $h_0,h_1,\dots,h_{m-1}\in\A$
such that $\Psi := \psi +
\lb{h_0+h_1\nu+\cdots+h_{m-1}\nu^{m-1},\pi_*}_S \in\Vect^1(\A^\nu)$ satisfies
$\Psi= 0 \,\mod\nu^m$. We then write $\Psi = \sum_{i\geq
  m}\Psi_i\nu^i$, with $\Psi_i\in\Vect^1(\A)$ for all $i\geq m$.
As $\Psi$ and $\psi$ differ from a $1$-coboundary of the Poisson
algebra $(\A^\nu,\pi_*)$, Equality (\ref{eq:formal1cocycle}) together
with the fact that $\Psi= 0 \,\mod\nu^m$ imply that 
\begin{equation}
0 = \lb{\psi,\pi_*}_S\,\mod\nu^{m+1} 
 = \lb{\Psi,\pi_*}_S \,\mod\nu^{m+1}
 = \lb{\Psi_m,\pi_0}_S \,\nu^m.
\end{equation}
We then have obtained that $\delta^1_{\pi_0}(\Psi_m) = -\lb{\Psi_m,\pi_0}_S =
0$ and, as $H^1(\A,\pi_0)=\lbrace 0\rbrace$, we have the existence of
an element $h_m\in\A$, such that $\Psi_m = \delta^0_{\pi_0}(h_m)$. This
can be written as follows :
$$
 -\lb{h_m,\pi_0}_S =  \Psi_m = \psi_m 
+ \renewcommand{\arraystretch}{0.7}
\sum\limits_{\begin{array}{c}\scriptstyle i+j=m\\\scriptstyle
    0\leq i\leq m-1\\\scriptstyle j\in\N\end{array}} \lb{h_i,\pi_j}_S,
$$
which is exactly $\psi_m = -\renewcommand{\arraystretch}{0.7}
\sum\limits_{\begin{array}{c}\scriptstyle i+j=m\\\scriptstyle
     i,j\in\N\end{array}} \lb{h_i,\pi_j}_S$. Using this and
(\ref{eq:formal1coboundaryk}), we obtain 
$$
\psi + \lb{h_0+h_1\nu+\cdots+h_{m-1}\nu^{m-1}+h_m\nu^m,\pi_*}_S = 0
\,\mod\nu^{m+1},
$$
which we wanted to show.
\end{proof}
\begin{rem}
We point out that Lemma \ref{lma:H1H1formal} is also valid if the first
Poisson cohomology spaces associated to $(A,\pi_0)$,
$(\A^\nu,\pi_*)$ or $\left(\A^\nu/\ideal{\nu^N},\pi_*
\,\mod\nu^N\right)$ are replaced by the $k$-th Poisson cohomology spaces
associated to these Poisson algebras. The proof is clearly analogous.
In fact, in the present paper, we will only need the
result as stated above. The
generic Poisson algebras which we will consider in dimension three, in
Section \ref{sec:dim3}, will have indeed a first Poisson cohomology space
which is zero and non-zero $k$-th Poisson cohomology spaces, for
$k\in\lbrace 0, 2, 3\rbrace$.
\end{rem}
Before the main result of this paragraph, let us give another lemma.
\begin{lma}\label{lma:xiformalcocycle}
Let $(\A,\pi_0)$ be a Poisson algebra.
Let us suppose that $\pi_* \sim \pi'_*$ are two equivalent formal deformations of
$\pi_0$. According to Lemma \ref{lma:equiv}, there exists an element
$\xi\in\Vect^1_0(\A^\nu)$ such that $\pi'_* = e^{\ad_\xi}(\pi_*)$.
If 
$$
\pi_* = \pi'_* \mod \nu^{N} \quad\hbox{ for some } N\in\N^*,
$$ 
then $\xi\,\mod\nu^N$ is a $1$-cocycle of
the Poisson algebra $\left(\A^\nu/\ideal{\nu^N},\pi_*
\,\mod\nu^N\right)$, i.e.,
$$
\lb{\xi,\pi_*}_S\,\mod\nu^N=0.
$$
\end{lma}
\begin{proof}
By hypothesis, we have the following equality : 
\begin{equation}\label{eqn:exp1}
\pi'_* = e^{\ad_\xi}(\pi_*)
= \pi_* + \sum_{k\in\N^*} \frac{1}{k!} 
\underbrace{\lb{\xi,\lb{\xi,\dots\,,\lb{\xi,\pi_*}_S\dots}_S}_S}
_{\hbox{$k$ brackets}}.
\end{equation}
We will prove the desired result by induction on $N\in\N^*$.
If $N=1$, then
the hypothesis becomes the trivial one $\pi_0=\pi_0$ and $\xi\,\mod\nu
= 0$ is trivially a $1$-cocycle of the Poisson algebra $(\A,\pi_0)$.

Now, suppose that $N\geq 1$ and suppose also that if
$\pi_*=\pi'_* \mod \nu^N$, then $\xi\,\mod\nu^N$ is a $1$-cocycle of
the Poisson algebra $\left(\A^\nu/\ideal{\nu^N},\pi_*
\,\mod\nu^N\right)$. Assume that $\pi_*=\pi'_* \mod \nu^{N+1}$, then
of course we have $\pi_*=\pi'_* \mod \nu^N$ and by induction
hypothesis, $\lb{\xi,\pi_*}_S\,\mod\nu^N = 0$. This last equality and Equation
(\ref{eqn:exp1}) lead to :
\begin{eqnarray*}
0 &=& \sum_{k\in\N^*} \frac{1}{k!} 
\overbrace{\lb{\xi,\lb{\xi,\dots\,,\lb{\xi,\pi_*}_S\dots}_S}_S}
^{\hbox{$k$ brackets}} \;\;\;\mod\nu^{N+1} \\
&=& \lb{\xi,\pi_*}_S \;\mod\nu^{N+1},
\end{eqnarray*}
which exactly implies that $\xi\,\mod\nu^{N+1}$ is a Poisson $1$-cocycle of
the Poisson algebra $\left(\A^\nu/\ideal{\nu^{N+1}},\pi_*
\,\mod\nu^{N+1}\right)$, hence the result.
\end{proof}
Now, let us give the main result of this paragraph.
\begin{prp}\label{prp:H1nul}
Let $(\A,\pi_0)$ be a Poisson algebra and
assume that its first Poisson cohomology space is zero:
$H^1(\A,\pi_0)=\lbrace 0\rbrace$.  Let us suppose that $\pi_* = \pi_0
+ \sum_{i\in\N^*}\pi_i\nu^i$ and $\pi'_*= \pi_0 +
\sum_{i\in\N^*}\pi'_i\nu^i$ (with $\pi_i,\pi_i'\in\Vect^2(\A)$, for $i\in\N^*$) 
are two equivalent formal deformations of
$\pi_0$. If 
$$
\pi_* = \pi'_* \mod \nu^{N} \quad\hbox{ for some }N\in\N^*,
$$
then there exists $\psi\in\Vect^1(\A)$ such that:
$$
\pi_N -\pi'_N = \delta^1_{\pi_0}(\psi).
$$ 
\end{prp}
\begin{proof}
Let us consider $(\A,\pi_0)$ a Poisson algebra. We
suppose that $\pi_*$ and $\pi'_*$ are two equivalent formal
deformations of $\pi_0$. According to Lemma \ref{lma:equiv}, there exists 
$\xi=\sum_{k\in\N^*}\xi_k\nu^k\in\Vect_0^1(\A^\nu)$ satisfying :
$\pi'_* = e^{\ad_\xi}(\pi_*)$. Assume that $\pi_* = \pi'_* \mod
\nu^{N}$ for some $N\in\N^*$. Then Lemma \ref{lma:xiformalcocycle}
implies that $\xi\,\mod\nu^N$ is a $1$-cocycle of
the Poisson algebra $\left(\A^\nu/\ideal{\nu^N},\pi_*
\,\mod\nu^N\right)$. By hypothesis, $H^1(\A,\pi_0)=\lbrace 0\rbrace$,
so that, according to the point (b) of Lemma \ref{lma:H1H1formal}, 
there exists an element $H\in\A^\nu$ such that $\X := \xi  +
\lb{H,\pi_*}_S\in\Vect^1(\A^\nu)$ satisfies $\X = 0
\,\mod\nu^N$. We then write $\X = \sum_{i\geq N}\X_i\nu^i$, with
$\X_i\in\Vect^1(\A)$ for all $i\geq N$. 
Now, because $\lb{\xi,\pi_*}_S = \lb{\X,\pi_*}_S$, we have :
\begin{eqnarray*}
\pi_* -\pi'_*\;\mod\nu^{N+1} 
&=& \pi_* - e^{\ad_\xi}(\pi_*)  \;\mod\nu^{N+1}
 = \pi_* - e^{\ad_\X}(\pi_*) \;\mod\nu^{N+1} \\
&=& -\lb{\X,\pi_*}_S \;\mod\nu^{N+1}
 =  -\lb{\X_N,\pi_0}_S \nu^N.
\end{eqnarray*}
We conclude that $\pi_N - \pi'_N = -\lb{\X_N,\pi_0}_S =
\delta^1_{\pi_0}(\X_N)$, with $\X_N\in\A$, which the desired result. 
\end{proof}
Combining Proposition \ref{magic} and Proposition \ref{prp:H1nul}, we
obtain the proposition \ref{thmintro} anounced in the introduction. In
particular, we obtain the
\begin{prp}\label{thm}
Let $(\A,\pi_0)$ be a Poisson algebra. 
Using the notation and under the hypotheses of Proposition \ref{magic} and
if, in addition, the first Poisson cohomology space $H^1(\A,\pi_0)$ is
\emph{zero}, then we have the following:
\begin{enumerate}
\item[(a)] For any formal deformation $\pi_*$ of $\pi_0$, there exists
a \emph{unique} element $\aa$ of $\S$, such that $\pi_*$ is
equivalent to $\pi_*^\aa$;
\item[(b)] For any $m$-th order deformation $\pi_{(m)}$ of $\pi_0$
($m\in\N^*$), there exists a \emph{unique} element
$\aa_{(m+1)}\in\S$ of the form 
$\aa_{(m+1)}=(a_1^k,a_2^k,\dots,a_{m}^k,0,0,\dots)_{k\in\K}$ (i.e.,
  $\aa_{(m+1)}=(a^k_n)_{{k\in\K}\atop{n\in\N^*}}$ with $a^k_n=0$, for
  every $k\in\K$ and $n\geq m+1$),  such that $\pi_{(m)}$ is
equivalent to $\pi_*^{\aa_{(m+1)}}$ modulo $\nu^{m+1}$, i.e., in $\A^{\nu}/\ideal{\nu^{m+1}}$.
\end{enumerate}
\end{prp}
\begin{proof}
The existence of the elements $\aa$ and $\aa_{(m+1)}$ are given by
Proposition \ref{magic}, we now study the unicity.
To do this, we point out that if
$\aa=(a_n^k)_{{k\in\K}\atop{n\in\N^*}}$ and
$\bb=(b_n^k)_{{k\in\K}\atop{n\in\N^*}}$ are two elements of $\S$, defining 
two different formal deformations $\pi_*^\aa$ and $\pi_*^\bb$ of the
form (\ref{defo_family}) and
$N:=\min\{n\in\N^*|\pi_n^\aa\not=\pi_n^\bb\}$, then
$\Psi^\aa_N = \Psi^\bb_N$ and
$\pi_N^\aa-\pi_N^\bb$ is an element of
$\bigoplus_{k\in\K}\F\,\vartheta_k$ which is a complementary of $B^2(\A,\pi_0)$ in
$Z^2(\A,\pi_0)$. 
According to Proposition \ref{prp:H1nul}, if $\pi_*^\aa$ and
$\pi_*^\bb$ were equivalent, then $\pi_N^\aa-\pi_N^\bb$ should be a
Poisson coboundary of $(\A,\pi_0)$ (an element of $B^2(\A,\pi_0)$), 
we then conclude that $\pi_*^\aa$ and $\pi_*^\bb$ can
not be equivalent.
\end{proof}
\begin{rem}
Notice that this result, and also the propositions \ref{magic} and
\ref{prp:H1nul}, could be stated in an associative or Lie
context, in a very analogous way (by replacing the Poisson cohomology
by the Hochschild or Chevalley-Eilenberg cohomology and the Schouten bracket by
the appropriate graded Lie algebra structure on the spaces of cochains).
\end{rem}
\section{Formal deformations of Poisson structures in dimension two
  and three}
\label{sec:dim3}

In this section, we consider a large family of Poisson structures on the
affine space of dimension three $\F^3$ and on singular surfaces in $\F^3$. We study their formal
deformations. Using the general results obtained in Section \ref{sec:part_case} and
the Poisson cohomology of these Poisson structures, obtained in \cite{AP},
we obtain an explicit expression of all their formal deformations, up to equivalence.
For more details about these Poisson
brackets and their Poisson cohomology, see \cite{AP}. 
As previously, $\F$ denotes an arbitrary field of characteristic zero.

\subsection{Poisson structures on $\F^3$ associated to a polynomial}
\label{sbsec:Pstr}

In this paragraph, we denote by $\A$ the polynomial algebra
$\A=\F[x,y,z]$. To each polynomial $\v\in\A$, one associates naturally a
Poisson structure $\PB_\v$ on $\A$, defined by the brackets:
\begin{eqnarray}
\label{bracket}
\pb{x,y}_{\v}=\pp{\varphi}{z},\quad \pb{y,z}_{\v}=\pp{\varphi}{x},
\quad \pb{z,x}_{\v}=\pp{\varphi}{y}.
\end{eqnarray}
It is indeed easy to show that the skew-symmetric biderivation $\PB_\v$,
explicitly given by:
\begin{eqnarray}
\label{bracket_phi}
\PB_{\v}=\pp{\v}{z}\;\pp{}{x}\wedge\pp{}{y}+\pp{\v}{x}\;\pp{}{y}\wedge\pp{}{z}
+\pp{\v}{y}\;\pp{}{z}\wedge\pp{}{x},
\end{eqnarray}
satisfies the Jacobi identity, i.e., equips the associative
commutative algebra $\A$ with a Poisson structure.
In the following, we will assume that $\v$ is a weight homogeneous
polynomial of (weighted) degree $\w(\v)\in\N$, i.e., that there exists
(unique) positive integers
$\w_1,\w_2,\w_3\in\N^*$ (the \emph{weights} of the variables $x$, $y$ and
$z$), without any common divisor, such that:
\begin{equation}\label{euler_form}
\w_1 \,x\, \pp{\v}{x} + \w_2\, y\, \pp{\v}{y} + \w_3\, z\, \pp{\v}{z}= \w(\v) \v.
\end{equation}
This equation is also called the \emph{Euler Formula}. If this weight
homogeneous polynomial $\v$ has a so-called isolated singularity (at
the origin), then the Poisson cohomology of the Poisson algebra
$(\A,\PB_\v)$ has been explicitly determined in
\cite{AP}. Recall that a weight homogeneous polynomial
$\v\in\F[x,y,z]$ is said to have an \emph{isolated singularity} (at
the origin) if the vector space
\begin{eqnarray}
\label{A_sing}
\A_{sing}(\v):=\F[x,y,z]/\ideal{\pp{\v}{x}, \pp{\v}{y}, \pp{\v}{z}}
\end{eqnarray}
is finite-dimensional. Its dimension is then denoted by $\mu$ and
called the \emph{Milnor number} associated to $\v$.  When $\F=\C$,
this amounts, geometrically, to saying that the surface
$\Fe_{\v}:\{\v=0\}$ has a singular point only at the origin.
In \cite{AP}, it has been shown that the singularity of $\v$ intervenes 
in the Poisson cohomology of $(\A,\PB_\v)$, with
$\A_{sing}(\v)$. In the following, we will see that it also appears 
in the formal deformations of $\PB_\v$.

In the following, the polynomial $\v$ will always be a weight homogeneous
polynomial with an isolated singularity. The corresponding weights of the three variables
($\w_1$, $\w_2$ and $\w_3$)
are then fixed and the weight homogeneity of any polynomial in $\F[x,y,z]$
has now to be understood as associated to these weights. We will also use
the fact, that, for $\A=\F[x,y,z]$, we have natural isomorphisms 
\begin{eqnarray}
\label{isomultider}
\Vect^0(\A)\simeq\Vect^3(\A)\simeq \A, 
\qquad \Vect^1(\A)\simeq\Vect^2(\A)\simeq \A^3,
\end{eqnarray}
chosen as 
$$
\renewcommand{\arraystretch}{1.3}
\begin{array}{ccc}
    \Vect^1(\A)&\longrightarrow& \A^3\\
       V &\longmapsto& (V[x],V[y],V[z]);
\end{array} \qquad\begin{array}{ccc}
    \Vect^2(\A)&\longrightarrow& \A^3\\
       V &\longmapsto& (V[y,z],V[z,x],V[x,y]);
\end{array} 
$$
and $\Vect^3(\A)\longrightarrow \A : V \longmapsto V[x,y,z]$.

\smallskip

The elements of $\A^3$ are viewed as vector-valued functions on $\A$,
so we denote them with an arrow, like $\vec{F}\in\A^3$. 
In $\A^3$, let $\cdot$, $\times$ denote respectively the usual inner and cross
products, while $\vec{\nabla}$, $\vec{\nabla}\times$, $\div$ denote
respectively the gradient, the curl and the divergence operators.
For example, with these notations and the above isomorphisms, the
skew-symmetric biderivation $\PB_{\v}$ is identified with the element
$\vn\v$ of $\A^3$.
Similarly, the so-called \emph{Euler derivation}
(associated to the weights of the variables), 
$\vec{e}_\w := \w_1 \,x\, \pp{}{x} + \w_2\, y\, \pp{}{y} + \w_3\, z\,
\pp{}{z}$ is viewed as the
element $\vec{e}_\w:=(\w_1\,x, \w_2\,y, \w_3\,z)\in\A^3$ and, with the
notations above, the Euler
formula~(\ref{euler_form}), for a weight homogeneous polynomial $F\in\A$ of
(weighted) degree $\w(F)$ becomes:
\begin{equation}\label{euler}
\vn F\cdot\vec{e}_\w = \w(F) \, F.
\end{equation}
\begin{rem}\label{rmk:coboundary}
Using the identifications above, it is possible to write the Poisson
coboundary operator, associated to $(\A,\PB_\v)$, in terms of elements in
$\A$ and elements in $\A^3$. Denoting this coboundary operator by
$\de_\v^k$, we obtain:
\begin{eqn}{delta_i}
\de^0_{\v}(F) &=& \vn F\times\vn\v, \quad \hbox{ for }F\in\A\simeq\Vect^0(\A),\\
\de^1_{\v}(\vec{F}) &=& -\vn(\vec{F}\cdot\vn\v)+\div(\vec{F})\vn\v, 
\quad\hbox{ for } \vec{F}\in\A^3\simeq \Vect^1(\A),\\
\de^2_{\v}(\vec{F}) &=& -\vn\v\cdot(\vn\times\vec{F}),
\quad\hbox{ for } \vec{F}\in\A^3\simeq \Vect^2(\A).
\end{eqn}%
From \cite{AP}, we know that, if $\v$ is a weight homogeneous
polynomial with an isolated singularity, then the Casimirs of the
Poisson algebra $(\A,\PB_\v)$ (i.e., the elements of the center of the
Poisson bracket, which are also the elements of $H^0(\A,\PB_\v)=\Ker
\de^0_\v$) are exactly the polynomials in $\v$.
\end{rem}
\subsection{The second Poisson cohomology space of $(\A,\PB_\v)$}
\label{H2}

We recall from~\cite{AP} that, as $\v\in\F[x,y,z]$ is a weight homogeneous
polynomial with an isolated singularity, the second Poisson cohomology space
associated to $(\A,\PB_\v)$ is given~by:
\begin{eqn}{h2old}
H^2(\A,\PB_\v)&\simeq& \renewcommand{\arraystretch}{0.7}
\ds\bigoplus
_{\begin{array}{c}\scriptstyle j=1\\\scriptstyle \w(u_j)\not=\w(\v)-|\w|\end{array}}^{\mu-1}
           \F[\v]\vn u_j
                       \oplus\renewcommand{\arraystretch}{0.7}
		   \ds\bigoplus_{\begin{array}{c}\scriptstyle j=0\\\scriptstyle \w(u_j)=\w(\v)-|\w|\end{array}}^{\mu-1}
           \F[\v]\,u_j\vn\v\\
       &&\qquad       \ds   \oplus \renewcommand{\arraystretch}{0.7}
                   \bigoplus_{\begin{array}{c}\scriptstyle j=1\\\scriptstyle \w(u_j)=\w(\v)-|\w|\end{array}}^{\mu-1}
           \F\vn u_j,
\end{eqn}%
where $|\w|:=\w_1+\w_2+\w_3$ denotes the sum of the weights of the
three variables and where the family $u_0:=1,
u_1,\dots,u_{\mu-1}\in\A$ is composed of weight homogeneous
polynomials in $\A$ whose images in $\A_{sing}(\v)$ give a basis of this
$\F$-vector space (and $u_0=1$).  In order to study the formal
deformations of the Poisson bracket $\PB_{\v}$, we need another basis
of $H^2(\A,\v)$. 
\begin{lma}
 \label{rewriting_h2}
If $\v\in\A=\F[x,y,z]$ is a weight homogeneous polynomial with an isolated
singularity, then the second Poisson
cohomology space associated to $(\A,\PB_{\v})$ is the $\F[\v]$-module:
\begin{eqnarray}\label{h2new}
H^2(\A,\PB_\v) &\simeq& 
\left\lbrace \renewcommand{\arraystretch}{1.5}
   \begin{array}{rcl}
     \ds\bigoplus_{j=0}^{\mu-1} \F[\v]\, u_j \vn\v \;\oplus\;
           \bigoplus_{j=1}^{\mu-1} \F\,\vn u_j,\qquad \hbox{ if }\w(\v)=|\w|,\\
 {}\\
     \ds\bigoplus_{j=1}^{\mu-1} \F[\v]\, u_j \vn\v \;\oplus\;
           \bigoplus_{j=1}^{\mu-1} \F\,\vn u_j, \qquad\hbox{ if }\w(\v)\not=|\w|,
   \end{array}
\right.\nonumber\\
 &\simeq& \bigoplus_{j\in\E_{\v}} 
           \F[\v]\, u_j \vn\v
                      \, \oplus\,
           \bigoplus_{j=1}^{\mu-1} \F\,\vn u_j,
\end{eqnarray}
where we have used the above notation and where
we have denoted by $\E_{\v}$, the set
$$
\renewcommand{\arraystretch}{1.5}
\E_{\v}:=
\left\lbrace
   \begin{array}{rcl}
     \{0,\dots,\mu-1\},&\hbox{if}&\w(\v)=|\w|,\\
     \{1,\dots,\mu-1\},&\hbox{if}&\w(\v)\not=|\w|.
   \end{array}
\right.
$$
\end{lma}
\begin{proof}
Using (\ref{delta_i}), we can compute, for all $i\in\N$ and $0\leq j\leq\mu-1$, 
\begin{eqnarray*}
\de_{\v}^1\Bigl(\v^i u_j \vec{e}_{\w}\Bigr) 
= \Bigl(\w(u_j)-\w(\v)+|\w|\Bigr)\v^i u_j\vn\v-\w(\v)\,\v^{i+1}\vn u_j.
\end{eqnarray*}
Now, using this equation, it is easy to verify that (\ref{h2old}) can also be
written as~(\ref{h2new}).
\end{proof}

\subsection{The formal deformations of $\PB_\v$}
\label{defo_dim3}

In this paragraph, our purpose is to consider the formal deformations of the Poisson
bracket $\PB_{\v}$ on $\F^3$, where $\v$ is a weight homogeneous
polynomial with an isolated singularity.
For this work, the Poisson cohomology that appears is the one associated to the Poisson
algebra $(\A=\F[x,y,z],\PB_{\v})$.

We first need to obtain a formula for the Schouten bracket of two specific
skew-symmetric biderivations of $\A$. In fact, for the study of the formal deformations of
$\PB_\v$, we will see that one only has to consider the skew-symmetric biderivations of the form 
$F\,\vn G\in\A^3\simeq\Vect^2(\A)$, with $F,G\in\A$.
Let us compute the Schouten bracket of two such skew-symmetric biderivations.
So let $F,G,H,L\in\A$. 
We compute the Schouten bracket $\lb{F\vn L,G\,\vn
  H}_S\in\Vect^3(\A)\simeq\A$, which we identify (according to (\ref{isomultider}))
  to its value $\lb{F\vn L,G\,\vn H}_S[x,y,z]\in\A$ and obtain:
\begin{equation}\label{eq:beleq}
\lb{F\,\vn L, G\,\vn H}_S=
F\,\vn L\cdot\left(\vn G\times\vn H\right)
+ G\,\vn H\cdot\left(\vn F\times\vn L\right).
\end{equation}
According to this equation, we have, for every $l,m\in\N$ and every $0\leq
i,j\leq \mu-1$,
\begin{eqnarray}
 \lb{\v^l u_i \vn\v, \v^m u_j \vn\v}_S = 0,\qquad  
 \lb{\vn u_i, \vn u_j}_S = 0,\label{equal1}
\end{eqnarray}
while, with the help of (\ref{eq:beleq}) and (\ref{delta_i}), we obtain,
\begin{equation}\label{eqn:last}
 \lb{\v^l u_i \vn\v,\vn u_j}_S 
= \delta^2_{\v}\left(\v^l u_i\vn u_j\right).
\end{equation}

The following proposition gives a formula for all formal deformations of
$\PB_\v$, up to equivalence.
\begin{prp}\label{prp:defo_A_v}
Let $\v\in\A=\F[x,y,z]$ be a weight homogeneous polynomial with an isolated
singularity. Consider the Poisson algebra $(\A,\PB_\v)$ associated
to~$\v$, where $\pi_0:=\PB_\v$ is the Poisson bracket given by
$$
\PB_\v=\pp{\v}{x}\; \pp{}{y}\wedge\pp{}{z} +
  \pp{\v}{y}\; \pp{}{z}\wedge\pp{}{x} +
  \pp{\v}{z}\; \pp{}{x}\wedge\pp{}{y}.
$$
Then we have the following:
\begin{enumerate}
\item[(a)] For all families of constants 
$\left(c^k_{l,i}\in\F\right)_{{(l,i)\in\N\times\E_\v}\atop{k\in\N^*}}$ 
and $\left(\bar c^{\,k}_r\in\F\right)_{{1\leq r\leq \mu-1}\atop{k\in\N^*}}$, such
  that, for every $k_0\in\N^*$, the sequences $(c^{k_0}_{l,i})_{(l,i)\in\N\times\E_{\v}}$ and
  $(\bar c^{\,k_0}_r)_{1\leq r\leq\mu-1}$ have finite supports, the formula 
\begin{eqnarray}\label{eq:defo_qcq}
\pi_* = \PB_\v + \sum_{n\in\N^*} \pi_n \nu^n,
\end{eqnarray}
where, for all $n\in\N^*$, $\pi_n$ is given by:
\begin{eqn}[2.3]{defo_form}
 \pi_n &=& \renewcommand{\arraystretch}{0.7}
 \ds\sum_{\begin{array}{c}\scriptstyle (l,i)\in\N\times\E_{\v}\\
 \scriptstyle 1\leq r\leq\mu-1\end{array}}
    \ds \sum_{\begin{array}{c}\scriptstyle a+b=n\\\scriptstyle a, b\in\N^*\end{array}} 
        c^a_{l,i}\,\bar c^{\,b}_r\, \v^l\, u_i\, \vn u_r  \\
&+& \renewcommand{\arraystretch}{0.7} 
\ds\sum_{(m,j)\in\N\times\E_{\v}} c^n_{m,j}\, \v^m\,u_j\vn\v 
\;+ \;\ds\sum_{1\leq s\leq\mu-1} \bar c^{\,n}_s\, \vn u_s,
\end{eqn}%
defines a formal deformation of $\PB_\v$, where the $u_j$ ($0\leq
j\leq \mu-1$) are weight homogeneous polynomials of $\A=\F[x,y,z]$, whose images in 
$\A_{sing}(\v)=\F[x,y,z]/\ideal{\pp{\v}{x},\pp{\v}{y},\pp{\v}{z}}$ give a
basis of the $\F$-vector space $\A_{sing}(\v)$ (and $u_0=1$).

\bigskip
\item[(b)] For any formal deformation $\pi_*'$ of $\PB_\v$,
there exist families of constants 
$\left(c^k_{l,i}\right)_{{(l,i)\in\N\times\E_\v}\atop{k\in\N^*}}$ 
and $\left(\bar c^{\,k}_r\right)_{{1\leq r\leq \mu-1}\atop{k\in\N^*}}$
(such that,
for every $k_0\in\N^*$, only a finite number of $c^{k_0}_{l,i}$ and
$\bar c^{\,k_0}_r$ are non-zero), 
for which $\pi_*'$ is equivalent to the
formal deformation $\pi_*$ given by the above formulas
\emph{(\ref{eq:defo_qcq})} and \emph{(\ref{defo_form})}.

\bigskip
\item[(c)] Moreover, if the (weighted) degree of the polynomial $\v$ is not equal
to the sum of the weights: $\w(\v)\not=|\w|$, then for any formal deformation $\pi_*'$ of
$\PB_\v$, there exist \emph{unique} families of constants 
$\left(c^k_{l,i}\right)_{{(l,i)\in\N\times\E_\v}\atop{k\in\N^*}}$ 
and $\left(\bar c^{\,k}_r\right)_{{1\leq r\leq \mu-1}\atop{k\in\N^*}}$
(with, for every $k_0\in\N^*$, only a finite number of non-zero $c^{k_0}_{l,i}$ and
$\bar c^{\,k_0}_r$),
such that $\pi_*'$
is equivalent to the formal deformation $\pi_*$ given by the  
formulas \emph{(\ref{eq:defo_qcq})} and \emph{(\ref{defo_form})}.

This means that
formulas \emph{(\ref{eq:defo_qcq})} and \emph{(\ref{defo_form})} give a system of
representatives for all formal deformations of $\PB_\v$, modulo equivalence.

\bigskip
\item[(d)] Analogous results hold if we replace formal
  deformations by $m$-th order deformations ($m\in\N^*$) and impose in
  \emph{(c)} that $c^k_{l,i}=0$ and $\bar c^{\,k}_r=0$, as soon as $k\geq m+1$. 
\end{enumerate}
\end{prp}
\begin{rem}
In particular, the previous proposition implies that, 
if $\w(\v)\not=|\w|$, the formal deformations of $\PB_\v$ defined by
(\ref{eq:defo_qcq}) and (\ref{defo_form}) and different from $\PB_\v$ (i.e.,
with some non all
zero constants $c^k_{l,i}\in\F$ and $\bar c^{\,k}_r\in\F$) are all
non-trivial formal deformations of $\PB_\v$ (i.e., non equivalent to $\PB_\v$).

\end{rem}
\begin{proof}
In fact, by proving the part (a) of the proposition, 
we will show that the Poisson algebra $(\A,\PB_\v)$ verifies
the hypotheses of Proposition \ref{magic}, with:
\begin{eqn*}[1.6]
        \K &=& \left(\N\times\E_\v\right)\cup\lbrace 1,\dots,\mu-1\rbrace,\\
       \aa &=& (c^a_{l,i},\,\bar c^{\,b}_r\,\mid\,(l,i)\in\N\times\E_\v, 1\leq
    r\leq\mu-1, 1\leq a,b\leq n)_{n\in\N^*}\in\S,\\
  \V_{r,j} &=& \v^r\, u_j \vn\v,\quad  (r,j)\in\N\times\E_\v,\\ 
      \V_i &=& \vn u_i, \quad 1\leq i\leq\mu-1,\\
\Psi^\aa_n &=& \renewcommand{\arraystretch}{0.7}
\ds\sum_{\begin{array}{c}\scriptstyle (l,i)\in\N\times\E_{\v}\\
 \scriptstyle r\in\{1,\dots,\mu-1\}\end{array}}\!
    \ds \sum_{\begin{array}{c}\scriptstyle a+b=n\\
 \scriptstyle a,b\in\N^*\end{array}}
        c^a_{l,i}\,\bar c^{\,b}_r\, \v^l\, u_i\, \vn u_r,
\end{eqn*}%
which implies part (b).
According to Proposition \ref{rewriting_h2}, the
elements $\v^r\, u_j \vn\v$ and $\vn u_i$, for $(r,j)\in\N\times\E_\v$, $1\leq
i\leq\mu-1$ give an $\F$-basis of the second Poisson cohomology space
$H^2(\A,\PB_\v)$ so that it suffices, for the parts (a) and (b) of the
proposition, to show that Equations
(\ref{eq:defo_qcq}) and (\ref{defo_form}) define a formal deformation of
$\pi_0=\PB_\v$.
Let us consider some constants $c^k_{l,i}\in\F$ and $\bar
c^{\,k}_r\in\F$, with $(l,i)\in\N\times\E_{\v}$, $1\leq r\leq\mu-1$ and $k\in\N^*$, and 
$\pi_*=\PB_\v + \sum_{k\in\N^*}\pi_k\nu^k$, with each $\pi_k$ given by:
\begin{eqn}[2.3]{defo_form'}
 \pi_k &=& \renewcommand{\arraystretch}{0.7}
 \ds\sum_{\begin{array}{c}\scriptstyle (l,i)\in\N\times\E_{\v}\\
  \scriptstyle r\in\{1,\dots,\mu-1\}\end{array}}
    \ds \sum_{\begin{array}{c}\scriptstyle a+b=k\\\scriptstyle a,b\in\N^*\end{array}} 
        c^a_{l,i}\,\bar c^{\,b}_r\, \v^l\, u_i\, \vn u_r  \\
&+& \renewcommand{\arraystretch}{0.7} 
\ds\sum_{(m,j)\in\N\times\E_{\v}} c^k_{m,j}\, \v^m\,u_j\vn\v 
\;+ \;\ds\sum_{s\in\{1,\dots,\mu-1\}} \bar c^{\,k}_s\, \vn u_s.
\end{eqn}%
(Notice that, for every $k_0\in\N^*$, only a finite number of $c^{k_0}_{l,i}$ and
$\bar c^{\,k_0}_r$ are non-zero.)
We have to verify (see Equation (\ref{eqn:eq_cohom})) that the following equation holds,
for every $n\in\N$,
\begin{equation}\label{def_eq_1'}
\delta^2_\v(\pi_{n+1})=\half\renewcommand{\arraystretch}{0.7}
  \sum_{\begin{array}{c}\scriptstyle i+j=n+1\\
  \scriptstyle i,j\geq1\end{array}}\lb{\pi_i,\pi_j}_S.
\end{equation}
For $n=0$, it becomes
$\delta^2_\v(\pi_1)=0$ and, according to (\ref{defo_form'}), we have 
$$
\pi_1 = \renewcommand{\arraystretch}{0.7} 
\ds\sum_{(m,j)\in\N\times\E_{\v}} c^1_{m,j}\, \v^m\,u_j\vn\v 
\;+ \;\ds\sum_{s\in\{1,\dots,\mu-1\}} \bar c^1_s\, \vn u_s,
$$
which is an element of $Z^2(\A,\PB_\v)$.
Now, assume that $n\geq 1$ and let us prove that the skew-symmetric
biderivations $\pi_1,\pi_2,\dots,\pi_{n+1}$, defined by (\ref{defo_form'}),
satisfy the equation (\ref{def_eq_1'}).  
By using (\ref{equal1}), one obtains that
$\half\renewcommand{\arraystretch}{0.7}
  \sum\limits_{\begin{array}{c}\scriptstyle i+j=n+1\\\scriptstyle
      i,j\geq1\end{array}}\lb{\pi_i,\pi_j}_S$ 
consists of six types of sums, listed here: 
\begin{eqnarray}
1/2\quad \sum c^{\,a}_{l,i}\; \bc^{\,b}_r\; c^{\,c}_{m,j}\; \bc^{\,d}_s && 
     \lb{\v^l u_i\vn u_r, \v^m u_j\vn u_s}_S,   \label{sum1}\\
1/2\quad\sum c^{\,a}_{l,i}\;\bc^{\,b}_r\; c^{\,q}_{m,j} &&
     \lb{\v^l u_i \vn u_r,\v^m u_j \vn\v}_S,    \label{sum2}\\ 
1/2\quad \sum c^{\,c}_{l,i}\;\bc^{\,d}_r\; c^{\,p}_{m,j} &&
     \lb{\v^l u_i \vn u_r,\v^m u_j \vn\v}_S,    \label{sum3}\\ 
1/2\quad \sum \bc^{\,q}_r\; c^{\,a}_{l,i}\; \bc^{\,b}_s &&
     \lb{\v^l u_i \vn u_r, \vn u_s}_S           \label{sum4}\\
1/2\quad \sum c^{\,c}_{l,i}\; \bc^{\,d}_r\; \bc^{\,p}_s &&
    \lb{\v^l u_i \vn u_r, \vn u_s}_S            \label{sum5}\\
1/2 \sum (c^{\,p}_{l,i}\; \bc^{\,q}_r + c^{\,q}_{l,i}\; \bc^{\,p}_r) &&
    \lb{\v^l u_i \vn\v, \vn u_r}_S              \label{sum6}
\end{eqnarray}
where the sums are taken over the $a,b,c,d,p,q,r,s,l,m,i,j\in\N$
satisfying:
$$
\begin{array}{cc}
p+q = n+1; & l,m\in\N\\
a+b = p;\quad c+d = q ; & i,j\in\E_{\v}\\
a,b,c,d,p,q\geq 1; & 1\leq r,s \leq\mu-1.
\end{array}
$$
One can observe that for all family of indices 
$(a,b,c,d,p,q,r,s,l,m,i,j)$, satisfying the conditions above, the indices
$(a',b',c',d',p',q',r',s',l',m',i',j')$, defined by:
$$
\begin{array}{rcccl}
p'=b+c, && a'=c, && i'=j,\\
q'=a+d, && b'=b, && j'=i,\\
r'=r,   && c'=a, && l'=m,\\
s'=s,   && d'=d, && m'=l,
\end{array}
$$
satisfy the same conditions, so that, in the first sum (\ref{sum1}), one
finds the element
\begin{equation}
 \label{elmt1}
c^{\,a}_{l,i}\; \bc^{\,b}_r\; c^{\,c}_{m,j}\; \bc^{\,d}_s\;
\lb{\v^l u_i\vn u_r, \v^m u_j\vn u_s}_S
\end{equation}
and the element
\begin{equation*}
c^{\,a'}_{l',i'}\; \bc^{\,b'}_{r'}\; c^{\,c'}_{m',j'}\; \bc^{\,d'}_{s'}\;
\lb{\v^{l'} u_{i'}\vn  u_{r'}, \v^{m'} u_{j'}\vn u_{s'}}_S.
\end{equation*}%
By definition of the primed indices, this second term is then
equal to the element $c^{\,a}_{l,i}\; \bc^{\,b}_r\; c^{\,c}_{m,j}\; \bc^{\,d}_s\;\lb{\v^m
  u_j\vn  u_r, \v^l u_i\vn u_s}_S$, whose sum with (\ref{elmt1})
is zero, according to~(\ref{eq:beleq}). 
This fact proves that the first sum (\ref{sum1}) is equal to zero.
With analogous arguments, one finds that the sums (\ref{sum2}),
(\ref{sum3}), (\ref{sum4}), (\ref{sum5}) are also zero.
We have then obtained that 
$\half\renewcommand{\arraystretch}{0.7}
  \sum\limits_{\begin{array}{c}\scriptstyle i+j=n+1\\\scriptstyle
      i,j\geq1\end{array}}\lb{\pi_i,\pi_j}_S$ 
is just given by the sum
(\ref{sum6}), that is to say:
\begin{eqnarray}\label{eq:cekilfo1}
\half\renewcommand{\arraystretch}{0.7}
\sum_{\begin{array}{c}\scriptstyle i+j=n+1\\
 \scriptstyle i,j\geq1\end{array}}\lb{\pi_i,\pi_j}_S &=& 
\half \!\!\!\sum_{\begin{array}{c}\scriptstyle (l,i)\in\N\times\E_{\v}\\
 \scriptstyle r\in\{1,\dots,\mu-1\}\end{array}}\!\!\!
  \sum_{\begin{array}{c}\scriptstyle p+q=n+1\\\scriptstyle p,q\in\N^*\end{array}}
   (c^{\,p}_{l,i}\; \bc^{\,q}_r + c^{\,q}_{l,i}\; \bc^{\,p}_r) 
    \lb{\v^l u_i \vn\v, \vn u_r}_S\nonumber\\
&=& 
\sum_{\begin{array}{c}\scriptstyle (l,i)\in\N\times\E_{\v}\\
 \scriptstyle r\in\{1,\dots,\mu-1\}\end{array}}\!\!\!
\sum_{\begin{array}{c}\scriptstyle p+q=n+1\\\scriptstyle p,q\in\N^*\end{array}}
  c^{\,p}_{l,i}\; \bc^{\,q}_r \;\; \delta^2_{\pi_0}\left(\v^l u_i\vn u_r\right),
\end{eqnarray}
where, for the second equality, we have used (\ref{eqn:last}).
Now, let us consider $\delta^2_{\pi_0}(\pi_{n+1})$. According to Equation
(\ref{defo_form'}), for $k=n+1$, and Lemma \ref{rewriting_h2}, 
\begin{eqn}[2.3]{eq:ckilfo}
 \pi_{n+1} &\in& \renewcommand{\arraystretch}{0.7}
 \ds\sum_{\begin{array}{c}\scriptstyle (l,i)\in\N\times\E_{\v}\\
  \scriptstyle r\in\{1,\dots,\mu-1\}\end{array}}
    \ds \sum_{\begin{array}{c}\scriptstyle p+q=n+1\\\scriptstyle p,q\in\N^*\end{array}} 
        c^p_{l,i}\,\bar c^q_r\; \v^l\, u_i\, \vn u_r + Z^2(\A,\PB_\v).
\end{eqn}%
Combining the equations (\ref{eq:cekilfo1}) and (\ref{eq:ckilfo}), 
we obtain that (\ref{def_eq_1'}) holds, hence
the first and second parts of the proposition. For the part
(c), we use Proposition 4.5 of \cite{AP} to obtain that, if
$\w(\v)\not=|\w|$, then $H^1(\A,\PB_\v)$ is zero and we conclude with
the help of Proposition \ref{thm}. Part (d) follows finally from the fact
that Propositions \ref{magic} and \ref{thm} are also valid for $m$-th order deformations.
\end{proof}
This proposition leads to the following result:
\begin{cor}\label{extension_v}
Let $\v\in\F[x,y,z]$ be a weight homogeneous polynomial with an isolated
singularity. Then, for all $m\in\N^*$, every $m$-th order deformation of
$\PB_\v$ extends to a $(m+1)$-th order deformation of $\PB_\v$.
\end{cor}
\begin{proof}
According to part (d) of Proposition \ref{prp:defo_A_v}, 
any $m$-th order deformation $\pi'_{(m)}$
of $\PB_\v$ is equivalent to an $m$-th order deformation of the form 
$\pi_{(m)} := \PB_\v + \sum_{n=1}^m \pi_n \nu^n$, where the
$\pi_n$ are defined as in (\ref{defo_form}). 
Let us denote by
$\Phi:\A^\nu/\ideal{\nu^{m+1}}\to\A^\nu/\ideal{\nu^{m+1}}$, the equivalence morphism
from $\pi_{(m)}$ to $\pi'_{(m)}$. Let us extend
$\Phi$ to an automorphism of~$(\A^\nu/\ideal{\nu^{m+2}},\cdot)$, in a natural way.

According to Proposition \ref{prp:defo_A_v}, we have that  
$\pi_{(m+1)} := \PB_\v + \sum_{n=1}^{m+1} \pi_n \nu^n$, where $\pi_{m+1}$
is defined with an analog of the formula (\ref{defo_form}), extends 
$\pi_{(m)}$ as an $(m+1)$-th order deformation. Then, the $(m+1)$-th order
deformation $\pi'_{(m+1)}$, defined by the formula 
$\pi'_{(m+1)}[F,G]
=\Phi\left(\pi_{(m+1)}[\Phi^{-1}(F),\Phi^{-1}(G)]\right) \mod\nu^{m+2}$ (for $F,G\in\A$
or $F,G\in\A^\nu/\ideal{\nu^{m+2}}$) extends $\pi'_{(m)}$ as an $(m+1)$-th order deformation.
\end{proof}
We point out that, in general, this property of extendibility of
deformations is not satisfied by an arbitrary
Poisson structure and the particular family of Poisson algebras
associated to weight homogeneous polynomials with an isolated singularity
$(\A,\PB_\v)$ has specific and nice properties of deformations.

\smallskip
Let us now consider the particular case where $\w(\v)=|\w|$, for which
we have 
$H^1(\A,\PB_\v)\simeq\F[\v]\vec{e}_\w$, according to Proposition 4.5 of \cite{AP}.
In this case, the part~(c) of Proposition \ref{prp:defo_A_v} and the
uniqueness of the constants $c^k_{l,i}$ and $\bar c^{\,k}_r$ do not
hold anymore. 
In particular, we will see that $\Phi=e^{\vec{e}_\w \nu}$, which is an
algebra morphism $\A^\nu\to\A^\nu$, equal to the identity modulo
$\nu$, is always an equivalence morphism between two different (except
in a very particular case) formal
deformations of the family given in Proposition \ref{prp:defo_A_v}. To
see that, assume $\w(\v)=|\w|$ and define $\xi := \vec{e}_\w \nu$ as being the element
$\xi = \w_1 x\,\nu\,\pp{}{x} + \w_2 y\,\nu\,\pp{}{y} + \w_3
z\,\nu\,\pp{}{z}\in\Vect^1_0(\A^\nu)$. Then take the formal deformation
$\pi_*$ of $\pi_0=\PB_\v$, given by two arbitrary families of constants 
$\left(c^a_{l,i}\right)_{{(l,i)\in\N\times\E_\v}\atop{a\in\N^*}}$ 
and $\left(\bar c^{\,b}_r\right)_{{1\leq r\leq \mu-1}\atop{b\in\N^*}}$
(with, for every $a_0,b_0\in\N^*$, only a finite number of non-zero $c^{a_0}_{l,i}$ and
$\bar c^{\,b_0}_r$)
and formulas (\ref{eq:defo_qcq}) and
(\ref{defo_form}) of Proposition \ref{prp:defo_A_v}. 
Let us denote by $\pi'_*$ the formal deformation of $\pi_0$ given by
$\pi'_* :=e^{\ad_\xi}(\pi_*)$. According to Lemma \ref{lma:equiv}, 
the deformation $\pi'_*$ is equivalent
to $\pi_*$ and $\Phi=e^{\xi}$ is an equivalence morphism from $\pi_*$
to $\pi'_*$.
Then a direct
computation (using Euler Formula
(\ref{euler})) shows that $\pi'_*$ is also given by 
$\pi'_*=\pi_0 + \sum_{n\in\N^*} \pi'_n\nu^n$, where, for all
$n\in\N^*$,
\begin{eqn*}[2.3]
 \pi'_n &=& \renewcommand{\arraystretch}{0.7}
 \ds\sum_{\begin{array}{c}\scriptstyle (l,i)\in\N\times\E_{\v}\\
 \scriptstyle s\in\{1,\dots,\mu-1\}\end{array}}
    \ds \sum_{\begin{array}{c}\scriptstyle a+b=n\\\scriptstyle a,b\in\N^*\end{array}} 
        c'^a_{l,i}\,\bar c'^{\,b}_s\, \v^l\, u_i\, \vn u_s  \\
&+& \renewcommand{\arraystretch}{0.7} 
\ds\sum_{(m,j)\in\N\times\E_{\v}} c'^n_{m,j}\, \v^m\,u_j\vn\v 
\;+ \;\ds\sum_{s\in\{1,\dots,\mu-1\}} \bar c'^{\,n}_s\, \vn u_s,
\end{eqn*}%
with, for $n\in\N^*$, $(l,i)\in\N\times\E_\v$ and $1\leq r\leq \mu-1$,
\begin{equation*}
\renewcommand{\arraystretch}{0.7} 
c'^n_{l,i} := \sum_{\begin{array}{c}\scriptstyle k+r=n\\\scriptstyle k,r\in\N^*\end{array}}
 \frac{1}{r!} c^k_{l,i} (|\w|(l-1)-\w(u_i))^r
\end{equation*}
and
\begin{equation*}
\renewcommand{\arraystretch}{0.7} 
\bar c'^{\,n}_s := \sum_{\begin{array}{c}\scriptstyle k+r=n\\\scriptstyle k,r\in\N^*\end{array}}
 \frac{1}{r!} \bar c^{\,k}_s (\w(u_s)-|\w|)^r.
\end{equation*}
Moreover, $\pi'_*=\pi_*$, if and only if, $c^n_{l,i}=0$, for all
$(l,i)\in(\N\times\E_{\v})-\{(0,0)\}$ and $\bar c^{\,k}_s =0$, for all
$1\leq s\leq\mu-1$ such that $\w(u_s)\not=|\w|$. So that,
$\pi'_*=\pi_*$ if and only if $\pi_*$ is of the form:
$$
\renewcommand{\arraystretch}{0.7} 
\pi_* = \pi_0 + \sum_{n\in\N^*}
\left(\sum_{\begin{array}{c}\scriptstyle a+b=n\\\scriptstyle a,b\in\N^*\end{array}}\!\!
\sum_{\begin{array}{c}\scriptstyle s=1\\\scriptstyle \w(u_s)=|\w|\end{array}}^{\mu-1}\!\!
c^a_{0,0}\,\bar c^{\,b}_s\,\vn u_s + c^n_{0,0}\,\vn\v + \!\!
\sum_{\begin{array}{c}\scriptstyle t=1\\\scriptstyle \w(u_t)=|\w|\end{array}}^{\mu-1}\!\!
\bar c^{\,n}_t\vn u_t
 \right)\nu^n,
$$
i.e., $\pi_*$ is a weight-homogeneous formal deformation of $\pi_0$ of
(weighted) degree equal to zero, 
in other words, each $\pi_n$ is a weight-homogeneous biderivation of
weighted degree equal to zero, for all $n\in\N$. (For more information about
weight-homogeneous biderivations, see \cite{PLV}).

\subsection{Properties of the formal deformations of $\PB_\v$}
\label{sbsec:prps_v}

As in Proposition~\ref{prp:defo_A_v}, we have obtained an explicit
expression for the formal deformations of the Poisson bracket
$\PB_\v$, we will now be able
to give some properties of these deformations, when $\v\in\F[x,y,z]$ is
supposed to be weight homogeneous with an isolated singularity. 
First, we obtain the following:
\begin{prp}\label{prp:writing_chi_v}
Let $\v\in\A=\F[x,y,z]$ be a weight homogeneous polynomial with an isolated
singularity. Consider the Poisson algebra $(\A,\PB_\v)$ associated
to~$\v$, where $\PB_\v$ is the Poisson bracket given by
$
\PB_\v=\pp{\v}{x}\; \pp{}{y}\wedge\pp{}{z} +
  \pp{\v}{y}\; \pp{}{z}\wedge\pp{}{x} +
  \pp{\v}{z}\; \pp{}{x}\wedge\pp{}{y}.
$ 
Then, for every formal deformation $\pi'_*$ of
$\PB_\v$, there exist 
$\chi^\nu, \v^\nu\in \A^\nu$, such that $\pi'_*$ is
equivalent to the formal deformation $\pi_*=\chi^\nu \,\vn {\v^\nu}$.
\end{prp}
\begin{proof}
According to Proposition \ref{prp:defo_A_v}, an arbitrary formal
deformation $\pi'_*$ of $\PB_\v$ is equivalent to a formal
deformation $\pi_*$, of the form:
\begin{eqnarray*}
\pi_* = \PB_\v + \sum_{n\in\N^*} \pi_n \nu^n,
\end{eqnarray*}
with $\pi_n$ given by (\ref{defo_form}) for all $n\in\N^*$,
where the elements $c^k_{l,i}$ and $\bar c^{\,k}_r$ (with
$k\in\N^*$, 
$(l,i)\in\N\times\E_{\v}$ and $1\leq r\leq\mu-1$) are constants in
$\F$ (and for every $a,b\in\N^*$, only a finite
number of non-zero $c^a_{l,i}$ and $\bar c^{\,b}_r$).
It is easy to verify that the elements of $\A^\nu$, defined by
\begin{eqnarray*}
\chi^\nu := 
1 + \sum_{a\in\N^*}\left(\sum_{(l,i)\in\N\times\E_{\v}} c^a_{l,i}\v^l u_i\right)\nu^a
\quad
\hbox{ and } \quad
\v^\nu := \v +\sum_{b\in\N^*}\left(\sum_{r=1}^{\mu-1} \bar c^{\,b}_r u_r  \right)\nu^b,
\end{eqnarray*}
satisfy the identity $\pi_*=\chi^\nu\,\vn\v^\nu
\in(\A^\nu)^3\simeq\Vect^2(\A^\nu)$, 
so that $\pi'_*$ is equivalent to a deformation of the desired form.
\end{proof}
\begin{rem}
It is easy to verify that, on $\F^3$, the multiplication of a Poisson
structure $\PB$ by any polynomial $\chi\in\F[x,y,z]$ gives another
Poisson structure $\chi\PB$. We point out that this fact is in general
not true in other dimensions. In particular, for every
$\chi,\v\in\F[x,y,z]$, the skew-symmetric biderivation $\chi\PB_\v$
(identified to $\chi\vn\v\in\A^3$) is a Poisson structure on
$\F^3$. In the previous proposition~\ref{prp:writing_chi_v}, we have
seen that, morally, if one deforms a Poisson structure of the family
$(\PB_\v\simeq \vn\v \,\mid\,\v\in\A)$, one obtains a Poisson
structure on $\A^\nu$ which belongs to the family
$(\chi^\nu\PB_{\v^\nu} \simeq\chi^\nu\vn\v^\nu \,\mid\,
\chi^\nu,\v^\nu\in\A^\nu)$.
\end{rem}
The following corollary gives another property verified by the formal
deformations of $\PB_\v$.
\begin{cor}\label{casimir_v}
Let $\v\in\A=\F[x,y,z]$ be a weight homogeneous polynomial with an isolated
singularity. Consider the Poisson algebra $(\A,\PB_\v)$ associated
to~$\v$. Every formal deformation of $\PB_\v$ admits a formal Casimir.
\end{cor}
\begin{proof}
First, let us consider a formal deformation of $\pi_0$, supposed to be of the form
$\pi_*=\chi^\nu\vn \v^\nu$, where $\chi^\nu,\v^\nu\in\A^\nu$ and let
us show that $\v^\nu$ is then a formal Casimir for $\pi_*$. 
Under the identifications $\Vect^2(\A^\nu) \simeq (\A^\nu)^3$ and 
$\Vect^1(\A^\nu) \simeq (\A^\nu)^3$, we
indeed have $\pi_*[\v^\nu,\,\cdot] = \left(\chi^\nu\vn\v^\nu\right) \times \vn\v^\nu$, which is
equal to zero, as, by writing $\chi^\nu=\sum_{i\in\N}\chi_i\nu^i$ and 
$\v^\nu:=\sum_{j\in\N}\v_j\nu^j$, where $\chi_i,\v_j\in\A$, we have:
\begin{equation*}
\left(\chi^\nu\vn\v^\nu\right) \times \vn\v^\nu 
= \sum_{i\in\N}\sum_{l\in\N}\chi_i \left(\sum_{j+k=l}\vn\v_j\times\vn\v_k\right) \nu^{i+l}
\end{equation*}
where, for each $l\in\N$, 
the sum $\ds\sum_{j+k=l}\vn\v_j\times\vn\v_k$ is equal to zero, because
$\vn\v_j\times\vn\v_k = - \vn\v_k\times\vn\v_j$. 
Now, according to Proposition \ref{prp:writing_chi_v}, any formal
deformation $\pi'_*$ of $\PB_\v$ is equivalent to a formal deformation of
the form $\pi_*=\chi^\nu \vn\v^\nu$, where
$\chi^\nu,\v^\nu\in\A^\nu$. Then, there exists a morphism of Poisson algebras 
$\Phi :(\A^\nu,\pi_*) \to (\A^\nu,\pi'_*)$ which is the
identity modulo~$\nu$. Thus, $\Phi$ is invertible and, for any $F\in\A^\nu$, we have
$$
\pi'_*[\Phi(\v^\nu),F] = \Phi\left(\pi_*[\v^\nu,\Phi^{-1}(F)]\right)=0.
$$
Hence the fact that
$\Phi(\v^\nu)$ is a formal
Casimir for~$\pi'_*$.
\end{proof}
\subsection{The case of singular surfaces in $\F^3$}
\label{sbssec:A_v}

In this last paragraph, we study singular surfaces in $\F^3$, equipped
with Poisson structures, as regular as possible and, as in the other cases
above, we give an explicit expression for all formal deformations of
these Poisson brackets, up to equivalence.

As previously, $\v\in\F[x,y,z]$ still denotes a weight
homogeneous polynomial with an isolated singularity and the weights of
the three variables $x,y,z$ are still denoted by $\w_1,\w_2,\w_3$,
while their sum is $|\w|=\w_1+\w_2+\w_3$. To such a polynomial,
one can associate a surface $\Fe_\v$ in $\F^3$ whose singular locus is exactly the
set $\lbrace \pp{\v}{x}=\pp{\v}{y}=\pp{\v}{z}=0\rbrace$. In fact, this
singular surface is given by the zero locus of $\v$,
$\Fe_\v:\lbrace\v=0\rbrace$. This affine space is equipped with its algebra
of regular functions $\A_\v:=\ds\frac{\F[x,y,z]}{\ideal{\v}}$. 

In Remark \ref{rmk:coboundary}, we pointed out
that $\v$ is a Casimir for the Poisson structure $\PB_\v$ defined in
(\ref{bracket_phi}), that is to say, is an element of the center of the bracket
$\PB_\v$. Hence, the Poisson bracket $\PB_\v$ goes to the quotient algebra
$\A_\v$ and it induces a bracket $\PB_{\A_\v}$ on $\A_\v$ that is obviously
a Poisson bracket. 

In this paragraph, our purpose is to study the formal deformations of this
Poisson structure. First, as proved in Proposition $5.2$ of \cite{AP}, we
have $\Vect^3(\A_\v)\simeq \lbrace
0\rbrace$, so that $H^3(\A_\v,\PB_{\A_\v})\simeq\lbrace 0\rbrace$ and,
according to the equations (\ref{eqn:eq_cohom}) which govern the
extendibility of deformations, every $m$-th order deformation
$\PB_{\A_\v} + \pi_1\nu +\cdots +\pi_m\nu^m$
of $\PB_{\A_\v}$ ($m\in\N^*$) extends to a $(m+1)$-th order
deformation
$\PB_{\A_\v} + \pi_1\nu +\cdots +\pi_m\nu^m + \pi_{m+1}\nu^{m+1}$, by
choosing for $\pi_{m+1}$, any Poisson $2$-cocycle of $(\A_\v,\PB_{\A_\v})$. 

In Proposition 5.6 of \cite{AP}, we have obtained that the family $\lbrace \wp(u_j \vn\v), 0\leq
j\leq \mu-1\mid \w(u_j)=\w(\v)-|\w|\rbrace$, where $\mu$ is the Milnor number
of $\v$ and $\wp:\F[x,y,z]\to\A_\v$ is the natural projection, gives an $\F$-basis of
the second Poisson cohomology space of $(\A_\v,\PB_{\A_\v})$. 
Since $H^3(\A_\v,\PB_{\A_\v})\simeq\lbrace 0\rbrace$, a simple case of
Proposition~\ref{magic}, in which the skew-symmetric biderivations
$\Psi_n^\aa$ can be chosen as being zero, leads to the following
result (also valid for $m$-th order deformations of $\PB_{\A_\v}$).
\begin{prp}\label{prp:defo_Av}
Let $\v\in\F[x,y,z]$ be a weight homogeneous polynomial with an isolated
singularity. 
Consider the Poisson algebra $(\A_\v,\PB_{\A_\v})$ and denote by 
$\K=\lbrace j\in\lbrace 0,\dots,\mu-1\rbrace\mid
\w(u_j)=\w(\v)-|\w|\rbrace$. We have the following:
\begin{enumerate}
\item For every family of constants
$\left(\a^n_j\in\F\right)_{{j\in\K}\atop{n\in\N^*}}$, the formula 
\begin{eqnarray}\label{eq:defo_qcq2'}
\pi_* = \PB_{\A_\v} + \sum_{n\in\N^*} \left(\renewcommand{\arraystretch}{0.7} 
\ds\sum_{\begin{array}{c}\scriptstyle j=0\\
  \scriptstyle \w(u_j)=\w(\v)-|\w|\end{array}}^{\mu-1} \a_j^n\, \wp(u_j \vn\v)\right) \nu^n
\end{eqnarray}
defines a formal deformation of $\PB_{\A_\v}$.

\bigskip
\item For any formal deformation $\pi'_*$ of $\PB_{\A_\v}$, 
there exists a family of constants 
$\left(\a^n_j\right)_{{j\in\K}\atop{n\in\N^*}}$,
such that $\pi'_*$ is equivalent to the formal
deformation $\pi_*$ given by the above formula \emph{(\ref{eq:defo_qcq2'})}.
\end{enumerate}
\end{prp}
\begin{rem}
According to Proposition 5.5 of \cite{AP}, we have 
$$
\renewcommand{\arraystretch}{0.7}
H^1(\A_\v,\PB_{\A_\v}) \simeq \bigoplus_{\begin{array}{c}\scriptstyle j=0\\
  \scriptstyle \w(u_j)=\w(\v)-|\w|\end{array}}^{\mu-1} \F\,\wp(u_j \vec{e}_\w),
$$
which is zero if and only if $H^2(\A_\v,\PB_{\A_\v})$ is also zero and,
according to the previous proposition \ref{prp:defo_Av}, all
formal deformations of $\PB_{\A_\v}$ are in this case trivial (i.e.,
equivalent to $\PB_{\A_\v}$). In the previous case, considered in Paragraph
\ref{defo_dim3}, we have considered the
algebra morphism $\Phi=e^{\vec{e}_\w\,\nu}$, in the
case the Euler derivation $\vec{e}_\w$ was defining a non-trivial cohomological class in
the first Poisson cohomology space. Here, the derivation $\vec{e}_\w$ defines such a non-trivial
class, if and only if, $\w(\v)=|\w|$, but, in this case, according to
Proposition \ref{prp:defo_Av}, all formal deformations of
$\PB_{\A_\v}$ are equivalent to a formal deformation of the form:
$$
\pi_*= \PB_{\A_\v} + \sum_{n\in\N^*}\a_0^n \,\wp(\vn \v)\nu^n,
$$
where $\a_0^n\in\F$, for all $n\in\N^*$, 
and the algebra morphism $\Phi=e^{\xi}$, defined above (with
$\xi:=\vec{e}_\w\nu$) is an equivalence morphism from such a
$\pi_*$ to
$$
\pi'_* := e^{\ad_\xi}(\pi_*) = \PB_{\A_\v},
$$
because $\lb{\vec{e}_\w,\vn\v}_S=0$. So that, if $\w(\v)=|\w|$, the
Poisson structure $\PB_{\A_\v}$ is rigid, i.e., all its formal
deformations are equivalent to $\PB_{\A_\v}$ itself.
\end{rem}
\begin{rem}
The limit case where the surface in $\F^3$ is the plane $\F^2$,
equipped with its algebra of polynomial functions $\F[x,y]$ is studied
in the same way. Every Poisson structure is in this case of
the form $\PB^\psi = \psi\,\pp{}{x}\we\pp{}{y}$, with $\psi\in\F[x,y]$.

In \cite{monnier}, one finds explicit bases for the Poisson cohomology spaces in
dimension two, for the germified case, while, in \cite{r_v}, one finds the
dimensions of the Poisson cohomology spaces of the Poisson variety
$(\F[x,y],\PB^\psi)$, in the algebraic setting.

We now suppose that the polynomial $\psi\in\F[x,y]$
is a weight homogeneous polynomial of (weighted) degree $\w(\psi)$,
associated to the weights of the two variables $x$ and $y$,
denoted respectively by $\w_1$ and $\w_2$.  The methods used in
\cite{monnier} can be applied in the algebraic context and in
particular permit to obtain, when
$\psi\in\F[x,y]$ is a weight homogeneous square-free polynomial, the
following:
\begin{equation}\label{eq:H2_F2}
H^2(\F[x,y],\PB^\psi) \simeq \;\F[x,y]_{N(\psi)}\;\PB^\psi 
    \oplus
    \;\frac{\ds\F[x,y]}{\ds\Bigl\langle\pp{\psi}{x},\pp{\psi}{y}\Bigr\rangle}\,
    \pp{}{x}\we\pp{}{y},
\end{equation}
where $\F[x,y]_{N(\psi)}$ is the $\F$-vector space of all weight
homogeneous polynomials in $\F[x,y]$, of (weighted) degree equal to
$N(\psi):=\w(\psi)-\w_1-\w_2$.
As in the case of the Poisson algebra $(\A_\v,\PB_{\A_\v})$, this
explicit basis leads to an explicit writing of the formal / $m$-th
order deformations of $\PB^\psi$.
\end{rem}
\section{Final Remarks}
\begin{enumerate}
\item We recall the result of M. Kontsevich, stated in the
  introduction and saying that, for a Poisson manifold $(M,\PB)$,
  there is a correspondence between the equivalence classes of the
  formal deformations of $\PB$ and those of the associative product of
  $\Fe(M)$, which have as a first order term the Poisson bracket
  $\PB$. 
Considering this, a natural
extension of the results given here would be to consider the equivalence
classes of the formal deformations of the associative algebra
$\A=\F[x,y,z]$ which have as first order term a Poisson bracket of the
form $\PB_\v$, with $\v\in\A$, and compare them to the equivalence classes
of the formal deformations of the Poisson structure $\PB_\v$, obtained
in this paper. 
We hope to come back to this in a future publication.
\item After obtaining these results of deformation of the Poisson
  structures of the form $\PB_\v$, $\v\in\F[x,y,z]$, B. Fresse pointed out to me that they
could come from a $L_\infty$-equivalence between two
$L_\infty$-algebras. This other point of view 
opens new perspectives of research, which we plan to explore in the future.
\item 
In their paper (\cite{Ginzburg-Etingof}), P. Etingof and V. Ginzburg consider ``deformations''
of Poisson algebras, but with the meaning that the associative product
\emph{and} the Poisson bracket are simultaneously deformed. To do
that, they use a notion of ``Poisson cohomology'' which is the one
defined in \cite{Fr1}, \cite{Fr2}, \cite{Ginzburg-Kaledin} 
and is different from the one used in \cite{AP} and in
the present paper. It would be interesting to compare the present paper with the
one of P. Etingof and V. Ginzburg. 
\end{enumerate}

\bibliographystyle{plain}
\bibliography{ref}

\begin{thebibliography}{10}

\bibitem{A_L}
Jacques Alev and Thierry Lambre.
\newblock Comparaison de l'homologie de {H}ochschild et de l'homologie de
  {P}oisson pour une d\'eformation des surfaces de {K}lein.
\newblock In {\em Algebra and operator theory (Tashkent, 1997)}, pages 25--38.
  Kluwer Acad. Publ., Dordrecht, 1998.

\bibitem{BFFLS}
Fran{\c c}ois Bayen, Mosh\'e Flato, Christian Fronsdal, Andr{\'e} Lichnerowicz,
  and Daniel Sternheimer.
\newblock Deformation theory and quantization. {I and II}.
\newblock {\em Ann. Physics}, 111(1):61--151, 1978.

\bibitem{MR2125456}
Martin Bordemann, Abdenacer Makhlouf, and Toukaiddine Petit.
\newblock D\'eformation par quantification et rigidit\'e des alg\`ebres
  enveloppantes.
\newblock {\em J. Algebra}, 285(2):623--648, 2005.

\bibitem{PHP}
Pantelis Damianou, Herv\'e Sabourin, and Pol Vanhaecke.
\newblock Transverse {P}oisson structures to adjoint orbits in semi-simple lie
  algebras.
\newblock {\em Pacific Journal of Mathematics}, 232(1):111--138, 2007.

\bibitem{zung}
Jean-Paul Dufour and Nguyen~Tien Zung.
\newblock {\em {P}oisson structures and their normal forms}, volume 242 of {\em
  Progress in Mathematics}.
\newblock Birkhäuser, 2005.

\bibitem{Ginzburg-Etingof}
Pavel Etingof and Victor Ginzburg.
\newblock Non-commutative del pezzo surfaces and calabi-yau algebras,
  ar{X}iv:q-alg/0709.3593.

\bibitem{Fr1}
Benoit Fresse.
\newblock Homologie de {Q}uillen pour les alg\`ebres de {P}oisson.
\newblock {\em C. R. Acad. Sci. Paris S\'er. I Math.}, 326(9):1053--1058, 1998.

\bibitem{Fr2}
Benoit Fresse.
\newblock Th\'eorie des op\'erades de {K}oszul et homologie des alg\`ebres de
  {P}oisson.
\newblock {\em Ann. Math. Blaise Pascal}, 13(2):237--312, 2006.

\bibitem{GHm}
Murray Gerstenhaber.
\newblock The cohomology structure of an associative ring.
\newblock {\em Ann. of Math. (2)}, 78:267--288, 1963.

\bibitem{Ginzburg-Kaledin}
Victor Ginzburg and Dmitry Kaledin.
\newblock {P}oisson deformations of symplectic quotient singularities.
\newblock {\em Adv. Math.}, 186(1):1--57, 2004.

\bibitem{MR1191570}
Viktor~L. Ginzburg and Jiang-Hua Lu.
\newblock {P}oisson cohomology of {M}orita-equivalent {P}oisson manifolds.
\newblock {\em Internat. Math. Res. Notices}, (10):199--205, 1992.

\bibitem{gw}
Viktor~L. Ginzburg and Alan Weinstein.
\newblock Lie-{P}oisson structure on some {P}oisson {L}ie groups.
\newblock {\em J. Amer. Math. Soc.}, 5(2):445--453, 1992.

\bibitem{gutt}
Simone Gutt.
\newblock Variations on deformation quantization.
\newblock In {\em Conf\'erence Mosh\'e Flato 1999, Vol. I (Dijon)}, volume~21
  of {\em Math. Phys. Stud.}, pages 217--254. Kluwer Acad. Publ., Dordrecht,
  2000.

\bibitem{hara}
Abdeljalil Haraki.
\newblock Quadratisation de certaines structures de {P}oisson.
\newblock {\em J. London Math. Soc. (2)}, 56(2):384--394, 1997.

\bibitem{Hochschild}
Gerhard Hochschild.
\newblock On the cohomology groups of an associative algebra.
\newblock {\em Ann. of Math. (2)}, 46:58--67, 1945.

\bibitem{huebs}
Johannes Huebschmann.
\newblock {P}oisson cohomology and quantization.
\newblock {\em J. Reine Angew. Math.}, 408:57--113, 1990.

\bibitem{konts}
Maxim Kontsevich.
\newblock Deformation quantization of {P}oisson manifolds.
\newblock {\em Lett. Math. Phys.}, 66(3):157--216, 2003.

\bibitem{PLV}
Camille Laurent-Gengoux, Anne Pichereau, and Pol Vanhaecke.
\newblock {\em An invitation to {P}oisson structures, monograph in
  preparation}.

\bibitem{lichne}
Andr{\'e} Lichnerowicz.
\newblock Les vari\'et\'es de {P}oisson et leurs alg\`ebres de {L}ie
  associ\'ees.
\newblock {\em J. Differential Geometry}, 12(2):253--300, 1977.

\bibitem{monnier}
Philippe Monnier.
\newblock {P}oisson cohomology in dimension two.
\newblock {\em Israel J. Math.}, 129:189--207, 2002.

\bibitem{MR1442493}
Nobutada Nakanishi.
\newblock {P}oisson cohomology of plane quadratic {P}oisson structures.
\newblock {\em Publ. Res. Inst. Math. Sci.}, 33(1):73--89, 1997.

\bibitem{Nowak}
Claus Nowak.
\newblock Star products for integrable {P}oisson structures on ${{\bf R}}^3$,
  ar{X}iv:q-alg/9708012.

\bibitem{MR1754236}
Michael Penkava and Pol Vanhaecke.
\newblock Deformation quantization of polynomial {P}oisson algebras.
\newblock {\em J. Algebra}, 227(1):365--393, 2000.

\bibitem{MR1849647}
Michael Penkava and Pol Vanhaecke.
\newblock Hochschild cohomology of polynomial algebras.
\newblock {\em Commun. Contemp. Math.}, 3(3):393--402, 2001.

\bibitem{APP}
Anne Pichereau.
\newblock Cohomologie de {P}oisson en dimension trois.
\newblock {\em C. R. Math. Acad. Sci. Paris}, 340(2):151--154, 2005.

\bibitem{AP}
Anne Pichereau.
\newblock {P}oisson (co)homology and isolated singularities.
\newblock {\em J. Algebra}, 299(2):747--777, 2006.

\bibitem{MR1264429}
Claude Roger, Mohamed Elgaliou, and Azeddine Tihami.
\newblock Une cohomologie pour les alg\`ebres de {L}ie de {P}oisson
  homog\`enes.
\newblock In {\em Publications du D\'epartement de Math\'ematiques. Nouvelle
  s\'erie}, volume 1990 of {\em Publ. D\'ep. Math. Nouvelle S\'er.}, pages
  1--26. Univ. Claude-Bernard, Lyon, 1990.

\bibitem{r_v}
Claude Roger and Pol Vanhaecke.
\newblock {P}oisson cohomology of the affine plane.
\newblock {\em J. Algebra}, 251(1):448--460, 2002.

\bibitem{stern}
Daniel Sternheimer.
\newblock Deformation quantization: twenty years after.
\newblock In {\em Particles, fields, and gravitation (L\'od\'z, 1998)}, volume
  453 of {\em AIP Conf. Proc.}, pages 107--145. Amer. Inst. Phys., Woodbury,
  NY, 1998.

\bibitem{vais}
Izu Vaisman.
\newblock {\em Lectures on the geometry of {P}oisson manifolds}, volume 118 of
  {\em Progress in Mathematics}.
\newblock Birkh\"auser Verlag, Basel, 1994.

\bibitem{pol}
Pol Vanhaecke.
\newblock {\em Integrable systems in the realm of algebraic geometry}, volume
  1638 of {\em Lecture Notes in Mathematics}.
\newblock Springer-Verlag, Berlin, second edition, 2001.

\bibitem{MR1321655}
Alan Weinstein.
\newblock Deformation quantization.
\newblock {\em Ast\'erisque}, (227):Exp.\ No.\ 789, 5, 389--409, 1995.
\newblock S\'eminaire Bourbaki, Vol.\ 1993/94.

\bibitem{xu2}
Ping Xu.
\newblock {P}oisson cohomology of regular {P}oisson manifolds.
\newblock {\em Ann. Inst. Fourier (Grenoble)}, 42(4):967--988, 1992.

\bibitem{xu1}
Ping Xu.
\newblock Gerstenhaber algebras and {BV}-algebras in {P}oisson geometry.
\newblock {\em Comm. Math. Phys.}, 200(3):545--560, 1999.

\end{thebibliography}
%


\end{document}